\documentclass[12pt,a4paper]{article}

\usepackage[british]{babel}

\usepackage[a4paper,top=2cm,bottom=2cm,left=2.5cm,right=2.5cm,marginparwidth=1.75cm]{geometry}

\usepackage[bibstyle=numeric, sorting=nyt, maxnames=99,minnames=1, giveninits=true, abbreviate=true, backend=bibtex]{biblatex} % APA 7th edition style citations using biblatex
\DeclareFieldFormat[article]{title}{#1,}
\DeclareFieldFormat[inproceedings]{title}{#1,}
\DeclareFieldFormat[book]{title}{#1,}
\AtEveryBibitem{\clearfield{number}}
\renewbibmacro{in:}{} 
\DeclareFieldFormat{pages}{#1}
\renewbibmacro*{journal+issuetitle}{%
  \iffieldundef{shortjournal}%
    {\usebibmacro{journal}}% 如果 shortjournal 未定义，使用 journal
    {\mkbibemph{\printfield{shortjournal}}}% 否则只使用 shortjournal
  \setunit{\addspace}%
  \usebibmacro{volume+number+eid}%
  \setunit*{\addspace}% 在 journal 和 volume 之间用空格
  \setunit{\addcolon\space}%
  \newunit}

\DeclareFieldFormat[article]{volume}{\mkbibbold{#1}} % 将 volume 加粗
\DeclareFieldFormat[article]{year}{(#1)} % 确保 year 格式正常

\renewbibmacro*{volume+number+eid}{%
  \printfield{volume}% 打印加粗的 volume
  \setunit*{\addspace}% 用空格分隔 volume 和 year
  \printfield{year}% 打印年份
  \setunit{\addspace}% 年份后空格
  \printfield{number}%
  \setunit{\addcomma\space}% 如果有 number，加逗号和空格
  \printfield{eid}}

\DeclareFieldFormat[book]{title}{\mkbibemph{#1}}
\DeclareFieldFormat{publisher}{#1}
\DeclareFieldFormat{address}{#1}
\renewbibmacro*{publisher+location+date}{%
  \iflistundef{publisher}
    {}
    {\addcomma\space\printlist{publisher}}%
  \iflistundef{location}
    {}
    {\addcomma\space\printlist{location}}%
  \setunit*{\addcomma\space}%
  \printfield{year}%
}

\usepackage{amsthm}  % 需要包含amsthm宏包

\newtheorem{definition}{Definition}  % 定义Definition环境
\newtheorem{theorem}{Theorem}  
\newtheorem{lemma}{Lemma} 
\newtheorem{remark}{Remark} 
 
\newtheorem{proposition}{Proposition}

\usepackage{amsmath}
\numberwithin{equation}{section}  % 让公式编号按章节计数
\usepackage{graphicx}
\usepackage{subcaption} 
\usepackage{tikz}
\usepackage[colorlinks=true, allcolors=blue]{hyperref}
\usepackage{hyperref}
\usepackage[title]{appendix}
\usepackage{mathrsfs}
\usepackage{amsfonts}
\usepackage{booktabs} % For \toprule, \midrule, \botrule
\usepackage{caption}  % For \caption
\usepackage{threeparttable} % For table footnotes
\usepackage[linesnumbered,ruled,vlined]{algorithm2e}
\usepackage{algpseudocode}  % For defining algorithms with pseudocode
\usepackage{algorithmicx}
\usepackage{algpseudocode}
\usepackage{listings}
\usepackage{enumitem}
\usepackage{chngcntr}
\usepackage{booktabs}
\usepackage{lipsum}
\usepackage{subcaption}
\usepackage{authblk}
\usepackage[T1]{fontenc}    % Font encoding
\usepackage{csquotes}       % Include csquotes
\usepackage{diagbox}
\usepackage{xcolor}

\usepackage{stmaryrd}

\newcommand{\vb}{\boldsymbol{v}}
\newcommand{\wb}{\boldsymbol{w}}
\newcommand{\ub}{\boldsymbol{u}}
\newcommand{\Hb}{\boldsymbol{H}}
\newcommand{\curlb}{\mathbf{curl}}
\newcommand{\Lb}{\boldsymbol{L}}
\newcommand{\Vb}{\boldsymbol{V}}
\newcommand{\Eb}{\boldsymbol{E}}

\newcommand{\nb}{\boldsymbol{n}}
\newcommand{\Jb}{\boldsymbol{J}}
\newcommand{\lb}{\boldsymbol{l}}
\newcommand{\xb}{\boldsymbol{x}}
\newcommand{\xib}{\boldsymbol{\xi}}
\newcommand{\etab}{\boldsymbol{\eta}}

\newcommand{\fb}{\boldsymbol{f}}
\newcommand{\zerob}{\boldsymbol{0}}
\newcommand{\Hcurlb}{\textbf{\textit{H}(curl)}}

\newcommand{\zetab}{\boldsymbol{\zeta}}
\newcommand{\qb}{\boldsymbol{q}}
\newcommand{\pb}{\boldsymbol{p}}
\newcommand{\Pib}{\boldsymbol{\Pi}}
\newcommand{\gb}{\boldsymbol{g}}
\newcommand{\zb}{\boldsymbol{z}}
\newcommand{\phib}{\boldsymbol{\phi}}

\usepackage{setspace}
\usepackage{titlesec}
\titleformat{\section} % Redefine section numbering format
  {\normalfont\Large\bfseries}{\thesection.}{1em}{}
  
\usepackage{lineno} % Add the lineno package

\rightlinenumbers % Right-align line numbers

\usepackage{float}   % for customizing caption position
\usepackage{caption} % for customizing caption format


\title{A Discontinuous Galerkin Method for \Hcurlb-Elliptic Hemivariational Inequalities}

\author[1,2,3]{Xiajie Huang\thanks{Email: \texttt{xj.huang@sjtu.edu.cn}}}

\author[1]{Fei Wang\thanks{The work of this author was partially supported by
			the National Natural Science Foundation of China (Grant No.\ 12171383). Email: \texttt{feiwang.xjtu@xjtu.edu.cn}}}
\author[4]{Weimin Han\thanks{The work of this author was partially supported by Simons Foundation Collaboration Grants (Grant No.\ 850737). Email: \texttt{weimin-han@uiowa.edu}}}
\author[5]{Min Ling \thanks{Email: \texttt{lingmin@imu.edu.cn}}}
\affil[1]{\small School of Mathematics and Statistics, Xi’an Jiaotong University, Xi’an, Shaanxi 710049, China}
\affil[2]{\small School of Mathematical Sciences, Shanghai Jiao Tong University, Shanghai 200240, China}
\affil[3]{Shanghai Artificial Intelligence Laboratory, Shanghai, China}
\affil[4]{Department of Mathematics, University of Iowa, Iowa City, IA 52242, USA}
\affil[5]{School of Mathematical Sciences, Inner Mongolia University, Hohhot, 010021, China}
\date{}  % Remove date

\begin{document}
\maketitle

\begin{abstract}  
%Due to widespread applications, research on high-temperature superconductors is always at the forefront of physics research.
%In 1962, C. P. Bean proposed a critical state model to describe high-temperature superconductors\textsuperscript{\cite{bean1964magnetization,bean1962magnetization}}. Han et al. extended it to a general non-monotonic case in their paper \cite{han2022numerical} and obtained a $\Hcurlb$-elliptic hemivariational inequality through a temporal semi-discretization. 

In this paper, we develop a Discontinuous Galerkin (DG) method for solving $\Hcurlb$-elliptic hemivariational inequalities. By selecting an appropriate numerical flux, we construct an  Interior Penalty Discontinuous Galerkin (IPDG) scheme. A comprehensive numerical analysis of the IPDG method is conducted, addressing key aspects such as consistency, boundedness, stability of the discrete formulation, and the existence, uniqueness, uniform boundedness of the numerical solutions. Building on these properties, we establish a priori error estimates, demonstrating the optimal convergence order of the numerical solutions under suitable solution regularity assumptions. Finally, a numerical example is presented to illustrate the theoretically predicted convergence order and to show the effectiveness of the proposed method.
  
\end{abstract}

%\textbf{Keywords}: keyword1, keyword2, keyword3, keyword4, keyword5, keyword6.  

%-------------------------------------------
% Paper Body
%-------------------------------------------
%--- Section ---%
%\section*{Nomenclature}

%\begin{tabbing}
%$T$\qquad \= Temperature (K)\\
%$u_i$ \> Velocity in the x-direction (m/s)\\
%$\tau_{ij}$ \> Shear stress (N/m2)\\
%$\omega$ \> Specific turbulent dissipation rate (1/s)\\
%$Y_\omega$ \> Dissipation of $\omega$
%\end{tabbing}

%--- Section ---%
\section{Introduction}

To describe the mixed state of type-II superconductors, in particular high-temperature superconductors, Bean proposed the critical-state theory \cite{bean1964magnetization,bean1962magnetization}. The basic principle of this model can be stated as follows. When a superconductor is in the mixed state, the magnitude of the current density $|\Jb|$ cannot exceed a critical value $g$. In the regions penetrated by the magnetic field, the current density is $|\Jb| = g$, and the electric field $\Eb$ is parallel to the current. When the magnitude of the current density $\Jb$ is strictly less than the critical value $g$, the electric field $\Eb = \zerob$. Mathematically, this can be expressed as
\begin{equation*}
|\Jb| \leq g; \quad |\Jb| < g \Rightarrow \Eb = \zerob; \quad |\Jb| = g \Rightarrow \Jb = \kappa \Eb \text{ for some } \kappa \geq 0,
\end{equation*}
One can eliminate the unknown parameter $\kappa$ and derive the following equivalent expression:
 \begin{equation}\label{eqn1.0}
 	|\Jb|\leq g, \quad \Jb\cdot\Eb=g|\Eb|.
 \end{equation}
With the use of the notion of convex subdifferential $\partial_c$, the relation can be compactly written as
 \begin{equation}\label{current}
 	\Jb\in\partial_c(g|\Eb|).
 \end{equation}
This is a nonsmooth monotone constitutive law, and the corresponding Maxwell models naturally lead to variational inequalities.

Modeling and analysis of variational inequalities of the Maxwell equations date back to Duvaut and Lions \cite{duvaut1976maxwell} in 1970s. In the context of superconductivities, mathematical models in the form of variational inequalities as extensions of Bean-type critical-state models were developed and studied in the early references \cite{bossavit1994numerical,bossavit1995superconductivity} and \cite{prigozhin1996bean}, and more recently in \cite{yousept2017hyperbolic, winckler2019fully, winckler2020adaptive, Yo20, MWY24}. Other references on Maxwell-type variational inequalities include \cite{yousept2020well} on a well-posedness theory for electromagnetic obstacle problems, \cite{yousept2021maxwell} on Maxwell quasi-variational inequalities with temperature and magnetic-field-dependent critical current, \cite{hensel2022numerical, hensel2023eddy} on numerical analysis of Maxwell obstacle problems in electric shielding and their eddy-current approximation, and \cite{caselli2023quasilinear} on quasilinear variational inequalities in ferromagnetic shielding.

More generally, the critical current density $g$ may depend on temperature and magnetic field strength. It has been shown in \cite{daeumling1990oxygen,cantoni2005anisotropic} that its dependence on the magnetic field strength is often non-monotonic. In such cases, the resulting model leads to Maxwell quasi-variational inequalities rather than standard variational inequalities; see, for example, \cite{yousept2021maxwell}. Another extension, motivated by non-monotonic constitutive behavior, was proposed in \cite{han2022numerical}, where the convex constitutive law~\eqref{current} is replaced by a Clarke-subdifferential relation
\begin{equation} \label{eqn.currentClarke}
\Jb \in \partial\psi(\Eb).
\end{equation}
Let $\Omega\subset\mathbb{R}^3$ be a bounded Lipschitz domain, the function $\psi(\xb, \Eb): \Omega \times \mathbb{R}^3 \rightarrow \mathbb{R}$ is locally Lipschitz continuous with respect to the variable $\Eb$ and this dependence is allowed to be non-convex. To simplify the notation, we write $\psi(\Eb)$ for $\psi(\xb, \Eb)$.  The symbol $\partial \psi(\Eb)$ denotes the Clarke subdifferential of $\psi$ with respect to the variable $\Eb$. Following the implicit Euler time discretization described in \cite{han2022numerical}, we consider an $\Hcurlb$-elliptic hemivariational inequality:  find $\Eb\in\Vb$ such that
\begin{equation*}
a(\Eb,\vb)+\int_{\Omega}\psi^0(\Eb;\vb)\,\mathrm{d}x\ge\langle\fb,\vb \rangle\qquad\forall\,\vb\in \Vb,
\end{equation*}
where
\begin{align*}
\Vb&:=\Hb_0(\curlb, \Omega),\\
a(\Eb,\vb)&:=\int_{\Omega}\epsilon\Eb\cdot\vb\,\mathrm{d}x+\int_{\Omega}\mu^{-1}(\nabla\times\Eb) \cdot (\nabla\times\vb)\,\mathrm{d}x.
\end{align*}
Here $\fb\in\Vb^*$, $\epsilon$ and $\mu$ denote the scaled permittivity and permeability, respectively, and $\psi^0$ is the generalized directional derivative in the sense of Clarke. This semidiscrete stationary problem is the object of the present paper.

Since analytical solutions of hemivariational inequalities are rarely available, their numerical approximations have attracted sustained attention, cf.\ \cite{haslinger1999finite} for a summary account of the early work on the finite element solution of hemivariational inequalities. In this paper, the term ``hemivariational inequalities'' is used to also mean the more general variational-hemivariational inequalities.  In the literature, the reference \cite{MR3284570} is the first paper that provides an optimal first-order error estimate for linear finite element approximations of a hemivariational inequality.   Representative contributions include \cite{han2017numerical,kalita2014semidiscrete,barboteu2015numerical} on finite element analyses for elliptic, parabolic, and hyperbolic HVIs, \cite{wang2020numerical,ling2024numerical} for numerical treatments of history-dependent problems, and \cite{jing2024finite,wang2024optimal} for related problems in fluid models. Recent overviews of the area may be found in \cite{han2019numerical,HFWH25}. Beyond standard finite elements, virtual element methods have been developed in both conforming and nonconforming settings \cite{feng2021virtual,feng2022nonconforming}. More recently, discontinuous Galerkin methods were proposed for elliptic HVIs in semi-permeable media in \cite{wang2020discontinuous} and further extended to contact-mechanics HVIs in \cite{wang2023discontinuous}.

For the specific $\Hcurlb$-elliptic hemivariational inequality with constitutive law $\Jb\in\partial \psi(\Eb)$,  a conforming edge finite element method is analyzed in \cite{han2022numerical} where an optimal-order error estimate is proved. On the other hand, discontinuous Galerkin discretizations for Maxwell equations are by now well established; see, e.g., \cite{perugia2003hp, houston2005interior, grote2007interior}. This makes it natural to ask whether interior penalty DG techniques can be adapted to the present nonconvex $\Hcurlb$ setting. To the best of our knowledge, a DG discretization together with a priori error analysis for the $\Hcurlb$-elliptic hemivariational inequality governed by $\Jb\in\partial\psi(\Eb)$ has not been reported.

In this paper, we develop an interior penalty discontinuous Galerkin method for the $\Hcurlb$-elliptic hemivariational inequality. The DG framework is attractive because of its flexibility in handling locally varying approximation spaces, interfaces, and boundary conditions. We prove the consistency, boundedness, and stability of the discrete formulation, establish existence and uniqueness of the discrete solution, derive a priori error estimates under suitable regularity assumptions, and present numerical experiments that confirm the predicted convergence order.

The rest of the paper is organized as follows. In Section \ref{sec2}, we review some definitions to be used later and recall the mathematical formulation of the $\Hcurlb$-elliptic hemivariational inequality. In Section \ref{sec3}, we present the DG discretization. In Section \ref{sec4}, we carry out an error analysis and derive a priori error estimates for the IPDG method applied to the $\Hcurlb$-elliptic hemivariational inequality. Finally, the last section reports numerical experiments that illustrate the theoretical convergence order established in this paper.

%--- Section ---%
\section{Preliminaries}\label{sec2}

We first recall the notions of the convex subdifferential and the Clarke (generalized) subdifferential.
Let $V$ be a Banach space and denote by $V^*$ its dual space. 

\begin{definition}
\label{Subdifferential}
Assume $\varphi: V\rightarrow \mathbb{R}\cup \{+\infty\}$
is a proper, convex, and lower semicontinuous function on $V$. The set
\[
\partial_c\varphi(u)=\{\xi\in V^*:\varphi(v)-\varphi(u)\geq \langle\xi, v-u\rangle\ \forall\, v \in V\}
\]
is called the convex subdifferential of the function $\varphi$ at $u\in V$. If $\partial_c\varphi(u)\neq \emptyset$, any element $\xi \in \partial_c\varphi(u)$ is called a subgradient of $\varphi$ at $u$.
\end{definition}

\begin{definition}\label{definition}
Assume $\psi: V\rightarrow\mathbb{R}$ is a locally Lipschitz continuous function. The generalized directional derivative of $\psi$ at $u\in V$ in the direction $v\in V$ is defined as
\[
\psi^0(u;v)=\limsup_{w\rightarrow u,\lambda \downarrow 0}\frac{\psi(w+\lambda v)-\psi(w)}{\lambda},
\]
and the Clarke subdifferential of $\psi$ at $u\in V$ is defined as
\[
\partial \psi(u)=\{\xi \in V^*: \psi^0(u;v)\geq \langle\xi,v\rangle\text{ } \forall  v\in V \}.
\]	
\end{definition}

For Clarke (or generalized) subdifferentials, the following properties hold~\cite{clarke1983nonsmooth,clarke1990optimization}:

\noindent (1) The generalized directional derivative can be obtained using the Clarke subdifferential:
\begin{equation}\label{eqnclarke1}
\psi^0(u;v)=\max \{\langle \xi,v\rangle:\xi\in\partial\psi(u) \} \quad \forall u,v\in V.
\end{equation}
(2) The generalized directional derivative is positively homogeneous and subadditive with respect to the direction variable:
\begin{equation}
\psi^0(u;\lambda v)=\lambda\psi^0(u;v) \quad\forall \lambda \geq 0, \text{ } u,v\in V,
\end{equation}
\begin{equation}\label{eq: subadditive}
\psi^0(u;v_1+v_2) \leq \psi^0(u;v_1)+\psi^0(u;v_2) \quad \forall u,v_1,v_2\in V.
\end{equation}
(3) Suppose $\psi_1,\psi_2: V\rightarrow \mathbb{R}$ are both locally Lipschitz continuous.  Then,
\begin{equation}\label{subaddition of Clarke}
\partial (\psi_1+\psi_2)(u) \subset \partial\psi_1(u)+\partial\psi_2(u) \quad \forall u\in V,
\end{equation}
equivalently,
\begin{equation}
(\psi_1+\psi_2)^0(u;v)\leq \psi_1^0(u;v)+\psi_2^0(u;v) \quad \forall u,v\in V.
\end{equation}

We will use Sobolev spaces for the formulation and analysis of the problem. We refer the reader to any standard reference on Sobolev spaces for further details; see, e.g., \cite{AF2003,Bre2011,evans2022partial}. Let $k\ge 0$ be an integer and $1\le p\le \infty$.  Given a domain $\Omega \subset \mathbb{R}^d$, $d\in \mathbb{N}$ being the spatial dimension, let $W^{k,p}(\Omega)$ denote the Sobolev space of $L^p(\Omega)$-functions whose weak derivatives of orders less than or equal to $k$ are in $L^p(\Omega)$, equipped with the standard norm $\|\cdot\|_{W^{k,p}(\Omega)}$. When $p=2$, we write $H^k(\Omega) = W^{k,2}(\Omega)$, with the corresponding norm denoted as $\|\cdot\|_{k,\Omega}$. In particular, for $k=0$, $H^0(\Omega) = L^2(\Omega)$, and its norm is written as $\|\cdot\|_{0,\Omega}$.  Sobolev spaces with vector-valued functions are denoted by boldface symbols, e.g., $\Lb^p(\Omega)$ and $\Hb^k(\Omega)$ are spaces of vector-valued functions with each component belonging to $L^p(\Omega)$ and $H^k(\Omega)$, respectively.  In a three-dimensional domain $\Omega$,
\begin{align*}
\Lb^p(\Omega) &= \{\ub=(u_1, u_2, u_3)^{\top}: u_i\in L^p(\Omega), i=1,2,3\},\\
\Hb^k(\Omega) &= \{\ub=(u_1, u_2, u_3)^{\top}: u_i\in H^k(\Omega), i=1,2,3\}.
%\boldsymbol{P}^s(\Omega) &= \{\ub=(u_1(\xb), u_2(\xb), u_3(\xb))| u_i(\xb)\in P^s(\Omega), i=1,2,3\}.
\end{align*}
The following integration by parts formula holds for any bounded Lipschitz domain $\mathcal{D}\subset \mathbb{R}^3$:
\begin{align}\label{eqnpar}
\int_\mathcal{D}  (\nabla\times \vb) \cdot\qb\mathrm{d}\xb=\int_\mathcal{D} \vb \cdot(\nabla\times\qb)\mathrm{d}\xb+\int_{\partial\mathcal{D}}(\vb\times\qb)\cdot\nb\mathrm{d}S,\quad \forall \vb, \qb \in \Hb^1(\mathcal{D}).
\end{align}
Here, $\vb\times\qb$ in the boundary integral term denotes the pointwise cross product of the traces of $\vb$ and $\qb$ on $\partial \mathcal{D}$. By the identity 
\[
\nabla\cdot (\vb\times \qb)=(\nabla\times\vb)\cdot\qb-\vb\cdot(\nabla\times\qb),
\]
the formula~\eqref{eqnpar} follows for smooth vector fields from the classical Gauss theorem. The extension to $\vb, \qb \in \Hb^1(\mathcal{D})$ is obtained by density of $\boldsymbol{C}^\infty(\overline{\mathcal{D}})$ in $\Hb^1(\mathcal{D})$, continuity of the trace operator $\tau: \Hb^1(\mathcal{D})\rightarrow\Hb^{1/2}(\partial\mathcal{D})$, and the continuous embedding $\Hb^{1/2}(\partial\mathcal{D})\hookrightarrow \Lb^2(\partial\mathcal{D})$.

For $\Omega \subset \mathbb{R}^3$, we define the following function spaces related to the curl operator:
\begin{align*}
\Hb(\curlb,\Omega)&=\left\{\vb\in\Lb^2(\Omega):\nabla\times \vb\in \Lb^2(\Omega)\right\},\\
\Hb_0(\curlb,\Omega)&=\left\{\vb\in\Hb(\curlb,\Omega): \nb\times \vb=\zerob \ \text{on }\partial\Omega\right\},
\end{align*}
where $\nb$ denotes the unit outward normal vector on $\partial \Omega$ and the curl operator $\nabla\times$ is defined by
\[
\nabla\times \vb=\left(\frac{\partial v_3}{\partial x_2}-\frac{\partial v_2}{\partial x_3},\frac{\partial v_1}{\partial x_3}-\frac{\partial v_3}{\partial x_1},\frac{\partial v_2}{\partial x_1}-\frac{\partial v_1}{\partial x_2}\right)^\top.
\]
The norm associated with these spaces is given by
\[
\|\vb\|_{\curlb, \Omega} = \left(\|\vb\|_{0,\Omega}^2 + \|\nabla\times \vb\|_{0,\Omega}^2\right)^{1/2}.
\]

The hemivariational inequality problem considered in this paper arises from the backward Euler semidiscretization in time of the hyperbolic Maxwell problem (\cite{han2022numerical}). Here we recall the assumptions and weak formulation of this problem; for more details we refer to \cite{han2022numerical}. Let $\Omega \subset \mathbb{R}^3$ be a bounded domain with a Lipschitz continuous boundary $\Gamma$, and let $\nb$ denote the unit outward normal vector on $\Gamma$, which exists a.e. Set
\[
\Vb=\Hb_0(\curlb,\Omega).
\]
In a superconductor, the current density $\Jb$ and the electric field $\Eb$ is assumed to satisfy the relation
\begin{equation} \label{eqnClarkecurrent}
\Jb \in \partial \psi(\Eb).
\end{equation}
Here $\psi:\Omega\times\mathbb{R}^3 \rightarrow \mathbb{R}$ is locally Lipschitz in its second argument, and $\psi(\Eb)$ stands for $\psi(\xb,\Eb)$. The notation $\partial \psi(\Eb)$ denotes the Clarke subdifferential of $\psi$ with respect to the second variable, and $\psi^0$ denotes the corresponding generalized directional derivative. We assume that $\psi$ satisfies the following conditions:
\begin{equation}\label{eqn1.other}
 \begin{cases}
 \text{(a) }\psi(\cdot,\boldsymbol{\xi})\ \text{is measurable in $\Omega$ for any $\boldsymbol{\xi}\in \mathbb{R}^3$, and $\psi(\cdot,\boldsymbol{0})\in L^1(\Omega)$.}\\
 		\text{(b) For a.e.\ $\boldsymbol{x}\in\Omega$, the function $\psi(\boldsymbol{x},\cdot)$ is locally Lipschitz continuous in $\mathbb{R}^3$.}\\
 		\text{(c) There exist constants $c_0,c_1\geq 0$ such that for a.e.\ $\xb\in\Omega$ and any $\xib\in\mathbb{R}^3$,}
 		\\
 		\quad\text{ }\text{ }\text{$|\etab|\leq c_0+c_1|\boldsymbol{\xi}| \quad \forall \etab \in \partial\psi(\boldsymbol{x},\boldsymbol{\xi})$.}\\
 		\text{(d) There exists a constant $m\geq 0$ such that for a.e.\ $\boldsymbol{x} \in \Omega$,} \\
 		\quad\text{ }\text{ } \psi^0(\xb,\xib_1;\xib_2-\xib_1)+\psi^0(\xb,\xib_2;\xib_1-\xib_2) \leq m|\xib_1-\xib_2|^2 \quad \forall \xib_1,\xib_2 \in \mathbb{R}^3.
 	\end{cases}
\end{equation}
Let $\tilde{\epsilon}$ and $\tilde{\mu}$ denote the permittivity and permeability in the time-dependent problem, and let $k>0$ be the time step. As in \cite{han2022numerical}, we define
\begin{equation}\label{eqn3.16} 
\epsilon = k^{-1}\tilde{\epsilon}, \quad \mu = k^{-1}\tilde{\mu}, \quad \epsilon_i = k^{-1}\tilde{\epsilon}_i, \quad \mu_i = k^{-1}\tilde{\mu}_i, \quad i = 0,1. 
\end{equation}
Following \cite{han2022numerical}, we assume that $\epsilon, \mu\in L^\infty(\Omega)$ and that
\[
0 < \epsilon_0 \leq \epsilon(\xb) \leq \epsilon_1 < \infty,
\qquad
0 < \mu_0 \leq \mu(\xb) \leq \mu_1 < \infty
\quad \text{for a.e. }\xb\in\Omega.
\]
For $\Eb,\vb\in \Vb$, define the bilinear form 
\begin{equation}\label{eq: bilinear form for a(.,.)}
a(\Eb,\vb)=\int_{\Omega}\epsilon\Eb\cdot\vb\,\mathrm{d}\xb+\int_{\Omega}\mu^{-1}(\nabla\times\Eb) \cdot (\nabla\times\vb)\,\mathrm{d}\xb.
\end{equation}
The hemivariational inequality under consideration is to find $\Eb\in\Vb$ such that.
\begin{equation}\label{eqn3.17}
\Eb\in\Vb,\quad a(\Eb,\vb)+\int_{\Omega}\psi^0(\Eb;\vb)\,\mathrm{d}\xb\ge\langle\fb,\vb \rangle\quad\forall\,\vb\in \Vb.
\end{equation}
where the load functional $\fb\in\Vb^*$ is induced by a vector field $\tilde{\lb}\in \Lb^2(\Omega)$, that is,
\[
\langle \fb, \vb \rangle=\int_{\Omega}\tilde{\lb}\cdot\vb\, \mathrm{d}\xb,
\qquad \forall\, \vb \in \Vb.
\]
The existence and uniqueness of a solution to \eqref{eqn3.17} were proved in \cite[Theorem 4.3]{han2022numerical}.

\section{Discontinuous Galerkin discretization}
\label{sec3}

In this section, we introduce an IPDG approximation of the hemivariational inequality \eqref{eqn3.17}. The construction follows the standard interior penalty DG discretization of the curl-curl Maxwell operator; see, e.g., \cite{perugia2003hp,houston2005interior,grote2007interior}.

Let $\mathcal{T}_h=\{K\}$ be a shape-regular partition of $\overline{\Omega}$ into tetrahedral or hexahedral elements. Denote by $h_K$ the diameter of an element $K$, and let
\[
h=\max_{K\in\mathcal{T}_h} h_K.
\]
We use $\mathcal{F}_h^\mathcal{I}$ to denote the set of all interior faces, $\mathcal{F}_h^\mathcal{B}$ the set of all boundary faces, and let $\mathcal{F}_h = \mathcal{F}_h^\mathcal{I} \cup \mathcal{F}_h^\mathcal{B}$.

Let $f \in \mathcal{F}_h^\mathcal{I}$ be the common face of two tetrahedral or hexahedral elements $K_1$ and $K_2$, and let $\nb_1$ and $\nb_2$ be the outward unit normal vectors on $f$ for $K_1$ and $K_2$, respectively, cf.\ Figure \ref{Fig5.1}. For a piecewise $\Hb^1$ vector-valued function $\vb$, namely,
\begin{equation*}
\left. \vb\right|_K\in \Hb^1(K) \qquad \forall K\in\mathcal{T}_h,
\end{equation*}
we denote by $\vb_1$ and $\vb_2$ the traces of $\vb$ on $f$ taken from $K_1$ and $K_2$, respectively. The tangential jump and the average of $\vb$ on $f$ are defined as
\begin{equation*}
	\llbracket\vb \rrbracket=\nb_1\times \vb_1 + \nb_2\times \vb_2,\quad \{\vb\}=\frac{\vb_1+\vb_2}{2}.
\end{equation*}
If $f \in \mathcal{F}_h^\mathcal{B}$, we define
\[ \llbracket\vb\rrbracket=\nb \times \vb, \quad \{\vb\}=\vb. \]
For later use, we note that for piecewise $\Hb^1$ vector fields $\vb$ and $\qb$ satisfying
\[
\vb|_K, \qb|_K\in\Hb^1(K) \qquad \forall K\in\mathcal{T}_h,
\]
where $\nb_K$ denotes the outward unit normal vector on $\partial K$, the following formula
\begin{equation}\label{eqn1.11}
\begin{aligned}
\sum_{K\in\mathcal{T}_h}\int_{\partial K}(\vb\times\qb)\cdot\nb_K\,\mathrm{d}S
&=
\sum_{f\in\mathcal{F}_h^\mathcal{I}}\int_f \llbracket\vb\rrbracket\cdot\{\qb\}\,\mathrm{d}S
-\sum_{f\in\mathcal{F}_h}\int_f \llbracket\qb\rrbracket\cdot\{\vb\}\,\mathrm{d}S \\
&=
\sum_{f\in\mathcal{F}_h}\int_f \llbracket\vb\rrbracket\cdot\{\qb\}\,\mathrm{d}S
-\sum_{f\in\mathcal{F}_h^\mathcal{I}}\int_f \llbracket\qb\rrbracket\cdot\{\vb\}\,\mathrm{d}S
\end{aligned}
\end{equation}
holds.

\begin{figure}[H]
	\centering
	\begin{tikzpicture}
		\draw (0.8,0) -- (3.2,2.0);
		\draw[dashed] (3.2,2.0) -- (1.3,2.5);
		\draw[dashed] (1.3,2.5) -- (0.8,0);
		\draw (0.8,0) -- (-1.5,0.8);
		\draw[dashed] (-1.5,0.8) -- (1.3,2.5);
		\draw (-1.5,0.8) -- (1.05,4) -- (3.2,2.0);
		\draw (1.05,4) -- (0.8,0);
		\draw[dashed] (1.05,4) -- (1.3,2.5);
		\node (A) at (0.4,2.0) {$K_1$};
		\node (B) at (2.0,2.0) {$K_2$};
		\node (C) at (1.1,2) {$f$};
		\draw[->] (1.15,2.4) -- (1.8,2.4)node[right]{$\boldsymbol{n}_1$};
		\draw[->] (0.95,1.2) -- (0,1.2)node[left]{$\boldsymbol{n}_2$};
	\end{tikzpicture}
	\caption{Two adjacent tetrahedral elements and a common face}
	\label{Fig5.1}
\end{figure}
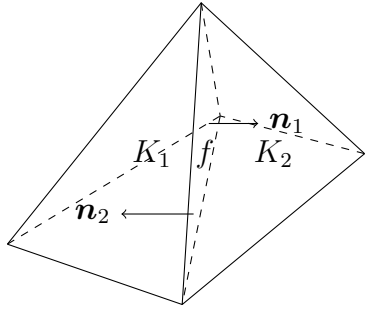

Given a positive integer $l$, we introduce the following DG space
\begin{equation*}
    \Vb_h=\{\vb\in\Lb^2(\Omega):\vb|_K\in\boldsymbol{P}^l(K)\ \forall\, K\in\mathcal{T}_h\}
\end{equation*}
where $\boldsymbol{P}^l(K)=\left[P^l(K)\right]^3$ and $P^l(K)$ is the space of polynomials of degree at most $l$ on $K$. For $\vb_h\in\Vb_h$, the piecewise curl operator $\nabla_h\times$ is defined elementwise by
\[
\nabla_h\times(\vb_h|_K)=\nabla\times(\vb_h|_K)\qquad \forall\,K\in\mathcal{T}_h.
\]

Numerical fluxes $\widehat{\Eb_h}$ and $\widehat{\pb_h}$ are used to approximate the traces of $\Eb$ and $\mu^{-1}\nabla\times\Eb$ on $\partial K$, respectively, while $\Eb_h$ is used to approximate $\Eb$ on $K$. We consider the DG method: find $\Eb_h\in\Vb_h$ such that
\begin{equation}\label{eqn1.15}
a_h(\Eb_h,\vb_h)+\int_\Omega\psi^0(\Eb_h;\vb_h)\mathrm{d}\xb\geq\langle \fb ,\vb_h\rangle\quad\forall\,\vb_h\in\Vb_h,
\end{equation}
where the bilinear form $a_h(\Eb_h,\vb_h)$ is defined as
\begin{equation}\label{eqn1.16}
\begin{aligned}
a_h(\Eb_h,\vb_h)=&\int_\Omega\epsilon\Eb_h\cdot\vb_h\mathrm{d}\xb+\int_\Omega\mu^{-1}(\nabla_h\times\Eb_h)\cdot(\nabla_h\times\vb_h)\mathrm{d}\xb
\\
&+\sum_{f\in\mathcal{F}_h}\int_f\llbracket\widehat{\Eb_h}-\Eb_h\rrbracket\cdot\{\mu^{-1}\nabla_h\times\vb_h\}\mathrm{d}S
\\
&-\sum_{f\in\mathcal{F}_h^{\mathcal{I}}}\int_f\llbracket\mu^{-1}\nabla_h\times\vb_h\rrbracket\cdot\{\widehat{\Eb_h}-\Eb_h\}\mathrm{d}S
\\
&+\sum_{f\in\mathcal{F}_h^{\mathcal{I}}}\int_f \llbracket\widehat{\pb_h}\rrbracket\cdot\{\vb_h\}\mathrm{d}S
-\sum_{f\in\mathcal{F}_h}\int_f\llbracket\vb_h\rrbracket\cdot\{\widehat{\pb_h}\}\mathrm{d}S.
\end{aligned}
\end{equation}
Define $h_f$ on each $f \in \mathcal{F}_h$ by
\begin{align*}
	h_f=\begin{cases}
		\min\{h_K,h_{K'}\}, \quad f\in\mathcal{F}_h^{\mathcal{I}}, \quad f =\partial K \cap \partial K',\\
		h_K, \qquad\qquad\ \,\,\,\,\,\,\,\,\,\,\,\,\,\,\,\, f \in \mathcal{F}_h^{\mathcal{B}},\quad f=\partial K \cap \partial \Omega.
	\end{cases}
\end{align*}
Let $\eta>0$ be a constant and define $\alpha$ on each $f \in \mathcal{F}_h$ by
\[ \alpha|_f=\eta h_f^{-1}.\]
Choose the numerical fluxes as follows:
\begin{equation*}
\begin{cases}
\widehat{\Eb_h}=\{\Eb_h\}\qquad\qquad\qquad\qquad\,\,\,\,\text{on } \mathcal{F}_h^{\mathcal{I}}, \\
{\nb\times\widehat{\Eb_h}}=\zerob \qquad\qquad\qquad\qquad\,\,\,\,\text{on } \mathcal{F}_h^{\mathcal{B}}, \\
\widehat{\pb_h}=\{\mu^{-1}\nabla_h\times\Eb_h\}-\alpha\llbracket\Eb_h\rrbracket \quad \text{on } \mathcal{F}_h.
\end{cases}
\end{equation*}
Using these choices in (\ref{eqn1.16}), we obtain the following bilinear form for the IPDG scheme:
\begin{align}\notag
	a_h(\Eb_h,\vb_h)&=\int_\Omega\epsilon \Eb_h\cdot \vb_h \mathrm{d}\xb+\int_\Omega\mu^{-1}(\nabla_h\times\Eb_h)\cdot(\nabla_h\times\vb_h)\mathrm{d}\xb
	\\\notag
	&\quad{}-\sum_{f\in\mathcal{F}_h}\int_f\llbracket\Eb_h\rrbracket\cdot\{\mu^{-1}\nabla_h\times\vb_h\}\mathrm{d}S-\sum_{f\in\mathcal{F}_h}\int_f\llbracket\vb_h\rrbracket\cdot\{\mu^{-1}\nabla_h\times\Eb_h\}\mathrm{d}S
	\\\label{eqn1.17}
	&\quad{}+\sum_{f\in\mathcal{F}_h}\int_f \alpha\llbracket\Eb_h\rrbracket\cdot\llbracket\vb_h\rrbracket\mathrm{d}S.
\end{align}
With the discrete bilinear form $a_h(\cdot, \cdot)$ defined by (\ref{eqn1.17}), the IPDG scheme to solve the $\Hcurlb$-elliptic hemivariational inequality~\eqref{eqn3.17} is to find $\Eb_h \in \Vb_h$ such that
\begin{equation}\label{eqn1.18}
a_h(\Eb_h,\vb_h)+\int_\Omega\psi^0(\Eb_h;\vb_h)\mathrm{d}\xb \geq\langle \fb,\vb_h\rangle\quad \forall \vb_h\in\Vb_h.
\end{equation}

\section{Error analysis}\label{sec4}

In this section, we present some properties of the IPDG scheme and provide a priori error estimates. On several occasions, we will apply Young's inequality with an arbitrarily small parameter $\varepsilon>0$:
\begin{equation}
a\,b\le \varepsilon\,a^2+\frac{1}{4\varepsilon}\,b^2\quad\forall\,a,b\in\mathbb{R}.
\label{mCS}
\end{equation}

\subsection{Properties of the IPDG scheme}

This subsection is devoted to the consistency, stability, and boundedness of the IPDG scheme, as well as the uniform boundedness of its numerical solution.

\begin{theorem}[Consistency]\label{thm: consistency}
Let $\Eb \in \Vb$ be a solution to the hemivariational inequality~\eqref{eqn3.17}. Assume further that
\begin{equation}\label{eq: consistency regularity assumption}
\mu^{-1} \nabla\times \Eb \in \Hb(\curlb, \Omega)
\end{equation}
and
\begin{equation}\label{eq: consistency regularity assumption2}
\left.\left(\mu^{-1}\nabla\times \Eb\right) \right|_K \in \Hb^1(K) \qquad \forall K\in \mathcal{T}_h.
\end{equation}
Then the following consistency relation holds:
\begin{equation}
a_h(\Eb,\vb_h)+\int_\Omega\psi^0(\Eb;\vb_h)\,\mathrm{d}\xb\ge\langle\fb,\vb_h\rangle
\quad \forall\,\vb_h\in\Vb_h.
\label{eq:cons}
\end{equation}
\end{theorem}
\begin{proof}
The bilinear form $a_h(\Eb,\vb_h)$ is obtained from (\ref{eqn1.17}) by replacing $\Eb_h$ with $\Eb$. Applying (\ref{eqn1.11}), we notice that the fourth term of $a_h(\Eb,\vb_h)$ is
\begin{equation}
  \begin{aligned}
    -\sum_{f\in\mathcal{F}_h}\int_f\llbracket\vb_h\rrbracket\cdot\{\mu^{-1}\nabla\times\Eb\}\mathrm{d}S
    =&-\sum_{K\in\mathcal{T}_h}\int_{\partial K}\mu^{-1}(\vb_h\times\nabla\times\Eb)\cdot \nb_K\mathrm{d}S
    \\
     &-\sum_{f\in\mathcal{F}_h^{\mathcal{I}}}\int_f\llbracket\mu^{-1}\nabla\times\Eb\rrbracket\cdot\{\vb_h\}\mathrm{d}S.
  \end{aligned}
\end{equation}
Since $\Eb \in \Vb$, the jump of $\Eb$ across $\mathcal{F}_h$ vanishes, namely,
\[
\llbracket \Eb \rrbracket = \zerob \qquad \text{on } \mathcal{F}_h.
\]
Furthermore, in view of the regularity assumption (\ref{eq: consistency regularity assumption}), we also have
\[
\llbracket \mu^{-1}\nabla\times\Eb \rrbracket = 0
\qquad \text{on } \mathcal{F}_h^{\mathcal{I}}.
\] Therefore,
\begin{align*}
a_h(\Eb,\vb_h) & =\int_\Omega\epsilon \Eb\cdot\vb_h\mathrm{d}\xb+\int_\Omega\mu^{-1}(\nabla\times\Eb)\cdot(\nabla_h\times\vb_h)\mathrm{d}\xb\\
&\quad{} -\sum_{K\in\mathcal{T}_h}\int_{\partial K}\mu^{-1}(\vb_h\times\nabla\times\Eb)\cdot \nb_K\mathrm{d}S.
\end{align*}
Since
\begin{equation*}
\left.\left(\mu^{-1}\nabla\times \Eb\right)\right|_K \in \Hb^1(K), \qquad \forall K\in\mathcal{T}_h,
\end{equation*}
we may apply the integration by parts formula (\ref{eqnpar}) on each element $K\in\mathcal{T}_h$:
\begin{align*}
\int_K\mu^{-1}(\nabla\times\Eb)\cdot(\nabla\times\vb_h)\mathrm{d}\xb
& =\int_K\nabla\times(\mu^{-1}\nabla\times\Eb)\cdot\vb_h\mathrm{d}\xb\\
&\quad{} -\int_{\partial K}\mu^{-1}((\nabla\times\Eb)\times\vb_h)\cdot\nb_K\mathrm{d}S.
\end{align*}
Thus,
\begin{equation*}
	a_h(\Eb,\vb_h)=\int_\Omega\epsilon \Eb\cdot\vb_h\mathrm{d}\xb+\int_\Omega\nabla\times(\mu^{-1}\nabla\times\Eb)\cdot\vb_h\mathrm{d}\xb.
\end{equation*}
% \textcolor{red}{Recalling the definition of the bilinear form $a(\cdot,\cdot)$ in~\eqref{eq: bilinear form for a(.,.)}, we obtain
% \[
% a_h(\Eb,\vb_h)=a(\Eb,\vb_h).
% \]
% Hence, invoking the hemivariational inequality~\eqref{eqn3.17}, we arrive at
% \begin{equation*}
% a_h(\Eb,\vb_h)+\int_\Omega\psi^0(\Eb;\vb_h)\,\mathrm{d}\xb\ge\langle\fb,\vb_h\rangle
% \quad \forall\,\vb_h\in\Vb_h.
% \end{equation*}
% This completes the proof.}

Since \(\Eb\in \Vb\) and \(\mu^{-1}\nabla\times \Eb\in \Hb(\curlb,\Omega)\), we may define
\[
\gb:=\epsilon \Eb+\nabla\times(\mu^{-1}\nabla\times \Eb)\in \Lb^2(\Omega).
\]
From the previous calculations, we already know that
\begin{equation}\label{eq:consistency-g}
a_h(\Eb,\vb_h)=\int_\Omega \gb\cdot \vb_h\,\mathrm d\xb
\qquad \forall\,\vb_h\in \Vb_h.
\end{equation}

Now let \(\phib\in \boldsymbol{C}_0^\infty(\Omega)\subset \Vb\). Since
\(\mu^{-1}\nabla\times \Eb\in \Hb(\curlb,\Omega)\), the integration-by-parts
formula
\[
\int_\Omega \left(\mu^{-1}\nabla\times\Eb\right) \cdot \left(\nabla\times\phib\right)\mathrm{d}\xb = \int_\Omega \nabla\times(\mu^{-1}\nabla\times \Eb)\cdot \phib\,\mathrm d\xb
\]
holds. Therefore, by the definition of the
bilinear form \(a(\cdot,\cdot)\) in~\eqref{eq: bilinear form for a(.,.)}, we have
\[
a(\Eb,\phib)
=
\int_\Omega \epsilon \Eb\cdot \phib\,\mathrm d\xb
+\int_\Omega \nabla\times(\mu^{-1}\nabla\times \Eb)\cdot \phib\,\mathrm d\xb
=
\int_\Omega \gb\cdot \phib\,\mathrm d\xb.
\]
Hence, by the hemivariational inequality~\eqref{eqn3.17},
\[
\int_\Omega \gb\cdot \phib\,\mathrm d\xb
+\int_\Omega \psi^0(\Eb;\phib)\,\mathrm d\xb
\ge
\langle \fb, \phib \rangle
=
\int_{\Omega}\tilde{\lb}\cdot\phib\, \mathrm{d}\xb
\qquad \forall\,\phib\in \boldsymbol{C}_0^\infty(\Omega).
\]

Next, we show that for a.e.\ $\xb\in\Omega$ and all $\zb_1,\zb_2\in\mathbb{R}^3$, it holds that
\begin{equation}\label{eq: Lipchitz estimate}
\bigl|\psi^0(\xb,\Eb(\xb);\zb_1)-\psi^0(\xb,\Eb(\xb);\zb_2)\bigr|
\le
(c_0+c_1|\Eb(\xb)|)\,|\zb_1-\zb_2|.
\end{equation}
Indeed, by the subadditivity of the Clarke generalized directional derivative with respect to the direction variable, taking
\[
u=\Eb(\xb), \qquad v_1=\zb_1-\zb_2, \qquad v_2=\zb_2
\]
in~\eqref{eq: subadditive}, we obtain
\begin{equation*}
\psi^0(\xb,\Eb(\xb);\zb_1)\leq \psi^0(\xb,\Eb(\xb);\zb_1-\zb_2)+\psi^0(\xb,\Eb(\xb);\zb_2),
\end{equation*}
which is equivalent to
\begin{equation}\label{eq: z1 z2 estimate1}
\psi^0(\xb,\Eb(\xb);\zb_1)-\psi^0(\xb,\Eb(\xb);\zb_2)\leq \psi^0(\xb,\Eb(\xb);\zb_1-\zb_2).
\end{equation}
Similarly, we have
\begin{equation}\label{eq: z1 z2 estimate2}
\psi^0(\xb,\Eb(\xb);\zb_2)-\psi^0(\xb,\Eb(\xb);\zb_1)\leq \psi^0(\xb,\Eb(\xb);\zb_2-\zb_1).
\end{equation}
Moreover, by the characterization of the generalized directional derivative in~\eqref{eqnclarke1} and assumption~\eqref{eqn1.other}(c), for any $\xib\in\mathbb{R}^3$,
\begin{equation}\label{eq: z1 z2 estimate3}
\psi^0(\xb,\Eb(\xb);\xib)
=
\max_{\etab \in \partial\psi(\xb,\Eb(\xb))} \etab \cdot \xib
\leq
\max_{\etab \in \partial\psi(\xb,\Eb(\xb))} |\etab|\,|\xib|
\leq
\left(c_0+c_1\left|\Eb(\xb)\right|\right)|\xib|.
\end{equation}
Combining \eqref{eq: z1 z2 estimate1}--\eqref{eq: z1 z2 estimate3}, we immediately deduce~\eqref{eq: Lipchitz estimate}.

Define the mapping
\[
J_{\Eb}(\wb):=\int_\Omega \psi^0(\Eb;\wb)\,\mathrm d\xb
\qquad \wb \in \Lb^2(\Omega).
\]
Since $\Eb\in\Vb$ and $\Omega$ is bounded, it follows that $c_0+c_1|\Eb|\in L^2(\Omega)$. For a.e.\ $\xb\in\Omega$ and any $\wb\in \Lb^2(\Omega)$, taking $\zb_1=\wb(\xb)$ and $\zb_2=\zerob$ in~\eqref{eq: Lipchitz estimate}, we obtain
\[
\bigl|\psi^0(\xb,\Eb(\xb);\wb(\xb))-\psi^0(\xb,\Eb(\xb);\zerob)\bigr|
\le
(c_0+c_1|\Eb(\xb)|)\,|\wb(\xb)|.
\]
Since $\psi^0(\xb,\Eb(\xb);\zerob)=0$, this reduces to
\[
\bigl|\psi^0(\xb,\Eb(\xb);\wb(\xb))\bigr|
\le
(c_0+c_1|\Eb(\xb)|)\,|\wb(\xb)|.
\]
The right-hand side belongs to $L^1(\Omega)$. Hence, by Hölder's inequality,
\[
\bigl|J_{\Eb}(\wb)\bigr|
\leq
\int_\Omega (c_0+c_1|\Eb(\xb)|)\,|\wb(\xb)| \,\mathrm{d}\xb
\leq
\left\|c_0+c_1\left|\Eb\right|\right\|_{0,\Omega}\left\|\wb\right\|_{0,\Omega}
<\infty.
\]
Therefore, $J_{\Eb}(\wb)$ is well defined. We next show that $J_{\Eb}$ is Lipschitz continuous on $\Lb^2(\Omega)$. For any $\wb_1,\wb_2\in\Lb^2(\Omega)$, we have
\[
\begin{aligned}
\left|J_{\Eb}(\wb_1)-J_{\Eb}(\wb_2)\right|
&\leq
\int_\Omega\left|\psi^0(\Eb;\wb_1)-\psi^0(\Eb;\wb_2)\right|\,\mathrm{d}\xb \\
&\leq
\int_\Omega (c_0+c_1|\Eb|)\,|\wb_1-\wb_2| \,\mathrm{d}\xb \\
&\leq
\left\|c_0+c_1\left|\Eb\right|\right\|_{0,\Omega}
\left\|\wb_1-\wb_2\right\|_{0,\Omega}.
\end{aligned}
\]
This shows that $J_{\Eb}$ is Lipschitz continuous on $\Lb^2(\Omega)$, with Lipschitz constant 
\[
\left\|c_0+c_1\left|\Eb\right|\right\|_{0,\Omega}.
\]

Since $\boldsymbol{C}_0^\infty(\Omega)$ is dense in $\Lb^2(\Omega)$, let $\{\phib_n\}\subset \boldsymbol{C}_0^\infty(\Omega)$ be such that
\[
\phib_n\to \wb \qquad \text{in } \Lb^2(\Omega).
\]
Since $\gb,\tilde{\lb}\in \Lb^2(\Omega)$ and $J_{\Eb}$ is Lipschitz continuous on $\Lb^2(\Omega)$, passing to the limit yields
\[
\int_\Omega \gb\cdot \wb\,\mathrm d\xb
+\int_\Omega \psi^0(\Eb;\wb)\,\mathrm d\xb
\ge
\int_\Omega \tilde{\lb}\cdot \wb\,\mathrm d\xb
\qquad \forall\,\wb\in \Lb^2(\Omega).
\]
Taking $\wb=\vb_h\in \Vb_h\subset \Lb^2(\Omega)$ and using \eqref{eq:consistency-g}, we obtain
\[
a_h(\Eb,\vb_h)+\int_\Omega\psi^0(\Eb;\vb_h)\,\mathrm d\xb
\ge
\int_\Omega \tilde{\lb}\cdot \vb_h\,\mathrm d\xb
=
\langle \fb,\vb_h\rangle,
\qquad \forall\,\vb_h\in \Vb_h,
\]
This completes the proof.
\end{proof}

Let
\[\Vb(h) = \Hb_0(\curlb; \Omega) + \Vb_h\]
with the seminorm and norm defined by
\begin{equation*}
{|\vb|_h^2=\sum_{K\in\mathcal{T}_h}\| \nabla_h\times\vb\|_{0,K}^2+\sum_{f\in\mathcal{F}_h}\| \alpha^{1/2}\llbracket\vb\rrbracket\|_{0,f}^2,\quad \|\vb\|_h^2=\|\vb\|_{0,\Omega}^2+|\vb|_h^2.}
\end{equation*}
The bilinear form $a_h(\cdot,\cdot)$ is initially defined on $\Vb_h\times\Vb_h$. To extend it to $\Vb(h)\times\Vb(h)$, we introduce an auxiliary bilinear form as follows \cite{grote2007interior, perugia2002hp}
\begin{align}\label{eqn2.2}
\widetilde{a}_h(\Eb,\vb)&=\int_\Omega\epsilon \Eb\cdot \vb \mathrm{d}\xb+\widetilde{b}_h(\Eb,\vb),
\end{align}
where
\begin{align*}
\widetilde{b}_h(\Eb,\vb)=&\sum_{K\in\mathcal{T}_h}\int_K\mu^{-1}(\nabla_h\times\Eb)\cdot(\nabla_h\times\vb)\mathrm{d}\xb
-\sum_{f\in\mathcal{F}_h}\int_f\llbracket\Eb\rrbracket\cdot\{\mu^{-1}\Pib_h(\nabla_h\times\vb)\}\mathrm{d}S
	\\
&-\sum_{f\in\mathcal{F}_h}\int_f\llbracket\vb\rrbracket\cdot\{\mu^{-1}\Pib_h(\nabla_h\times\Eb)\}\mathrm{d}S+\sum_{f\in\mathcal{F}_h}\int_f\alpha\llbracket\Eb\rrbracket\cdot\llbracket\vb\rrbracket\mathrm{d}S.
\end{align*}
Here, $\Pib_h$ is the $L^2$-projection from $\Lb^2(\Omega)$ onto $\Vb_h$. Note that $\widetilde{a}_h(\cdot,\cdot)$ coincides with $a_h(\cdot,\cdot)$ on $\Vb_h\times\Vb_h$.

\begin{lemma}[Boundedness]\label{lemma:boundedness}
There is a constant $C_b>0$ such that
\begin{equation}
|\widetilde{a}_h(\Eb,\vb)|\leq C_b\|\Eb\|_h \|\vb\|_h\quad \forall\,\Eb,\vb\in\Vb(h).
\label{ah:bd}
\end{equation}
\end{lemma}
\begin{proof}
Using H\"older's inequality, we obtain
\begin{equation*}
\begin{aligned}
\int_K\mu^{-1}(\nabla_h\times\Eb)\cdot(\nabla_h\times\vb)\mathrm{d}\xb
&\leq \mu_0^{-1}\int_K|\nabla_h\times\Eb|\cdot|\nabla_h\times\vb|\mathrm{d}\xb\\
&\leq \mu_0^{-1}\|\nabla_h\times \Eb \|_{0,K}\|\nabla_h\times \vb \|_{0,K}.
\end{aligned}
\end{equation*}
Similarly,
\begin{align*}
\int_f \alpha\llbracket \Eb \rrbracket \cdot \llbracket \vb \rrbracket \mathrm{d}S
&\leq  \| \alpha^{1/2}\llbracket\Eb\rrbracket\|_{0,f}\| \alpha^{1/2}\llbracket\vb\rrbracket\|_{0,f},\\
\int_\Omega\epsilon \Eb\cdot \vb \mathrm{d}\xb & \leq\epsilon_1\|\Eb\|_{0,\Omega}\|\vb\|_{0,\Omega}.
\end{align*}
Now let us bound the third term of the bilinear form $\widetilde{a}_h(\cdot,\cdot)$, which is a modification of the proof in \cite[Lemma 4]{grote2007interior}.
\begin{equation*}
\begin{aligned}
&\sum_{f\in\mathcal{F}_h}\int_f\llbracket\Eb\rrbracket\cdot\{\mu^{-1}\Pib_h(\nabla_h\times\vb)\}\mathrm{d}S
\\
\leq &\mu_0^{-1}\sum_{f\in\mathcal{F}_h}\int_f\left|\alpha^{1/2}\llbracket\Eb\rrbracket\right|\cdot\left|\alpha^{-1/2}\{\Pib_h(\nabla_h\times\vb)\}\right|\mathrm{d}S   \\
\leq &\mu_0^{-1}\sum_{f\in\mathcal{F}_h} \|\alpha^{1/2}\llbracket\Eb\rrbracket\|_{0,f} \cdot \|\alpha^{-1/2}\{\Pib_h(\nabla_h\times\vb)\}\|_{0,f}
\\
\leq &\mu_0^{-1}\left(\sum_{f\in\mathcal{F}_h} \int_f\alpha\left|\llbracket\Eb\rrbracket\right|^2 \mathrm{d}S\right)^{1/2} \cdot
\left(\sum_{f\in\mathcal{F}_h} \int_f\alpha^{-1} \left| \{\Pib_h(\nabla_h\times\vb)\} \right|^2 \mathrm{d}S\right)^{1/2}
\\
= &
\mu_0^{-1}\eta^{-1/2}\left(\sum_{f\in\mathcal{F}_h} \|\alpha^{1/2}\llbracket\Eb\rrbracket\|^2_{0,f}\right)^{1/2} \cdot\left(\sum_{f\in\mathcal{F}_h} \int_f  h_f\left| \{\Pib_h(\nabla_h\times\vb)\} \right|^2 \mathrm{d}S\right)^{1/2}.
\end{aligned}
\end{equation*}
Using the definition of $h_f$, we obtain
\begin{equation*}
\begin{aligned}
\sum_{f\in\mathcal{F}_h} \int_f  h_f \left| \{\Pib_h(\nabla_h\times\vb)\} \right|^2 \mathrm{d}S&\leq \frac{1}{2}\sum_{K\in\mathcal{T}_h}\int_{\partial K}h_K  |\Pib_h(\nabla_h\times \vb)|^2 \mathrm{d}S
\\
& \leq \frac{1}{2}\sum_{K\in\mathcal{T}_h}h_K\| \Pib_h(\nabla_h\times \vb)\|^2_{0,\partial K}.
\end{aligned}
\end{equation*}
Recalling the discrete trace theorem (\cite[Chapter 12.2, Lemma 12.8]{ern2021finite})
\[
\| \wb \|^2_{0,\partial K} \leq \tilde{C}^2 h^{-1}_K \| \wb\|_{0,K}^2
\quad \forall\, \wb\in\boldsymbol{P}^l(K),\ \forall\, K\in\mathcal{T}_h,
\]
where the positive constant $\tilde{C}$ depends only on the mesh regularity, the polynomial degree $l$, and the spatial dimension. From the $L^2$-projection property,
\[
\| \Pib_h \wb \|_{0,K} \leq \|\wb \|_{0,K} \quad \forall\, \wb \in \Lb^2(K).
\]
Combining the above results, we obtain
\begin{equation*}
\frac{1}{2}\sum_{K\in\mathcal{T}_h}h_K\|  \Pib_h(\nabla_h\times \vb)\|^2_{0,\partial K}\leq \frac{1}{2}\tilde{C}^2\sum_{K\in\mathcal{T}_h} \|\nabla_h\times \vb\|^2_{0,K}.
\end{equation*}
Finally, we obtain the bound
\begin{equation}
\begin{aligned}
\sum_{f\in\mathcal{F}_h}\int_f\llbracket\Eb\rrbracket\cdot\{\mu^{-1}\Pib_h(\nabla_h\times\vb)\}\mathrm{d}S
\leq &
\mu_0^{-1}(2\eta)^{-1/2}\tilde{C} \left(\sum_{f\in\mathcal{F}_h} \|\alpha^{1/2}\llbracket\Eb\rrbracket\|^2_{0,f}\right)^{1/2}
\\
&
\cdot \left(\sum_{K\in\mathcal{T}_h} \|\nabla_h\times \vb\|^2_{0,K} \right)^{1/2}.
\end{aligned}
\end{equation}
Similarly, we can bound the fourth term of the bilinear form as
\begin{equation}
\begin{aligned}
\sum_{f\in\mathcal{F}_h}\int_f\llbracket\vb\rrbracket\cdot\{\mu^{-1}\Pib_h(\nabla_h\times\Eb)\}\mathrm{d}S
\leq &
\mu_0^{-1}(2\eta)^{-1/2}\tilde{C} \left(\sum_{f\in\mathcal{F}_h} \|\alpha^{1/2}\llbracket\vb\rrbracket\|^2_{0,f}\right)^{1/2}
\\
&
\cdot \left(\sum_{K\in\mathcal{T}_h} \| \nabla_h\times \Eb\|^2_{0,K} \right)^{1/2}.
\end{aligned}
\end{equation}
Combining these results, we obtain \eqref{ah:bd}.
\end{proof}

\begin{lemma}[Stability]\label{lemma:stability}
Assume $\eta > \max\{1,\mu_1^2\}\,\tilde{C}^2/(2\,\mu_0^2)$.  There is a constant $C_s>0$ such that
\begin{equation}
\widetilde{a}_h(\Eb,\Eb)\geq C_s \|\Eb\|_h^2, \quad \forall \Eb\in\Vb(h).
\label{eqn2.5}
\end{equation}
\end{lemma}
\begin{proof}
\[
\int_\Omega\epsilon \Eb\cdot \Eb \mathrm{d}\xb \geq \epsilon_0\|  \Eb\|_{0,\Omega}^2 ,
\]
\[	\sum_{K\in\mathcal{T}_h}\int_K\mu^{-1}(\nabla_h\times\Eb)\cdot(\nabla_h\times\Eb)\mathrm{d}\xb \geq\mu_1^{-1}\sum_{K\in\mathcal{T}_h} \|\nabla_h\times \Eb \|_{0,K}^2,
\]
\[
\sum_{f\in\mathcal{F}_h}\int_f \alpha\llbracket\Eb\rrbracket\cdot\llbracket\Eb\rrbracket\mathrm{d}S=\sum_{f\in\mathcal{F}_h}\| \alpha^{1/2}\llbracket\Eb\rrbracket\|_{0,f}^2.
\]
Similar to the proof of Lemma \ref{lemma:boundedness}, we have
\begin{align*}
&-2\sum_{f\in\mathcal{F}_h}\int_f\llbracket\Eb\rrbracket\cdot\{\mu^{-1}\Pib_h(\nabla_h\times\Eb)\}\mathrm{d}S
\\
\geq&-2\mu_0^{-1}(2\eta)^{-1/2}\tilde{C}\left(\sum_{f\in\mathcal{F}_h}\| \alpha^{1/2}\llbracket\Eb\rrbracket\|_{0,f}^2\right)^{1/2} \cdot \left( \sum_{K\in\mathcal{T}_h}\|\nabla_h\times \Eb\|_{0,K}^2 \right)^{1/2}
\\
\geq&-\mu_0^{-1}(2\eta)^{-1/2}\tilde{C}\left(\sum_{f\in\mathcal{F}_h}\| \alpha^{1/2}\llbracket\Eb\rrbracket\|_{0,f}^2+\sum_{K\in\mathcal{T}_h}\|\nabla_h\times \Eb\|_{0,K}^2 \right).
\end{align*}
Therefore, when $\eta > \max\{1,\mu_1^2\}\,\tilde{C}^2/(2\,\mu_0^2)$, $\widetilde{a}_h(\Eb,\Eb)$ is bounded from below by
\begin{equation}\label{lower bound of tiled{a}}
\begin{aligned}
&\epsilon_0\|  \Eb\|_{0,\Omega}^2 + \left(1-\mu_0^{-1}(2\eta)^{-1/2}\tilde{C}\right)\sum_{f\in\mathcal{F}_h}\| \alpha^{1/2}\llbracket\Eb\rrbracket\|_{0,f}^2 \\
&+\left(\mu_1^{-1}-\mu_0^{-1}(2\eta)^{-1/2}\tilde{C}\right)\sum_{K\in\mathcal{T}_h}\|\nabla_h\times \Eb\|_{0,K}^2.
\end{aligned}
\end{equation}
Denote $C_0 = 1-\mu_0^{-1}(2\eta)^{-1/2}\tilde{C}$, $C_1 = \mu_1^{-1}-\mu_0^{-1}(2\eta)^{-1/2}\tilde{C}$.  Then,
\[ \widetilde{a}_h(\Eb,\Eb) \geq C_s\left(\sum_{f\in\mathcal{F}_h}\| \alpha^{1/2}\llbracket\Eb\rrbracket\|_{0,f}^2+\sum_{K\in\mathcal{T}_h}\|\nabla_h\times \Eb\|_{0,K}^2+\| \Eb\|_{0,\Omega}^2 \right) \] holds, where $C_s = \min\{\epsilon_0, C_0, C_1\}$.
\end{proof}

The next result establishes the uniform boundedness of the numerical solution defined by the IPDG scheme for the $\Hcurlb$-elliptic hemivariational inequality. We first record two consequences of the assumptions on $\psi$. From (\ref{eqnclarke1}) and (\ref{eqn1.other})\,(c), we have
\begin{equation}\label{eqn1.other2}
\left|\psi^0(\xib_1;\xib_2)\right| \leq \left(c_0+c_1|\xib_1|\right)|\xib_2| \quad \forall \xib_1,\xib_2 \in \mathbb{R}^3.
\end{equation}
Moreover, (\ref{eqn1.other})\,(d) is equivalent to (\cite[p.\ 124]{SM2018} or \cite[Proposition 2.42]{Han2024})
\begin{equation}
(\etab_1-\etab_2)\cdot (\xib_1-\xib_2) \geq -m|\xib_1-\xib_2|^2 \quad \text{a.e. } \xb \in \Omega,\ \forall\,\xib_i \in \mathbb{R}^3,\ \etab_i \in \partial \psi(\xb,\xib_i),\ i=1,2.
\label{property of psi 5}
\end{equation}

\begin{lemma} 
Assume \eqref{eqn1.other}, $m<\epsilon_0$ and $\eta > \max\{1,\mu_1^2\}\,\tilde{C}^2/(2\,\mu_0^2)$. If $\Eb_h\in \Vb_h$ is a solution to the problem \eqref{eqn1.18}, then $\|\Eb_h\|_h$ is uniformly bounded with respect to the mesh size $h$. 
\end{lemma}
\begin{proof}
By setting $\vb_h=-\Eb_h$ in (\ref{eqn1.18}), we obtain
\begin{equation}\label{eqn2.6}
	a_h(\Eb_h,\Eb_h) \leq \int_\Omega\psi^0(\Eb_h;-\Eb_h) \mathrm{d}\xb +\langle \fb,\Eb_h \rangle.
\end{equation}
From assumption (\ref{eqn1.other})\,(d), we have
\begin{equation*}
	\psi^0(\Eb_h;\zerob-\Eb_h) + \psi^0(\zerob;\Eb_h-\zerob) \leq m|\Eb_h|^2,
\end{equation*}
Using (\ref{eqn1.other2}), we obtain
\[  -\psi^0(\zerob;\Eb_h)\leq c_0|\Eb_h|; \]
hence,
\begin{equation*}
\int_\Omega\psi^0(\Eb_h;-\Eb_h) \mathrm{d}\xb \leq m\|\Eb_h\|_{0,\Omega}^2 + \int_\Omega c_0|\Eb_h| \mathrm{d}\xb.
\end{equation*}
Moreover,
\[ \langle \fb,\Eb_h\rangle\leq\| \fb \|_{\Vb^*}\| \Eb_h\|_{\curlb,\Omega}\leq \| \fb \|_{\Vb^*} \|\Eb_h\|_h.\]
By the Cauchy-Schwarz inequality, 
\[
\int_\Omega c_0 |\Eb_h|\mathrm{d}\xb \leq c_0 |\Omega|^{1/2}\| \Eb_h \|_{0,\Omega},
\]
where $|\Omega|$ means the Lebesgue measure of the bounded domain $\Omega$. Combining these inequalities with the lower bound (\ref{lower bound of tiled{a}}) of $\widetilde{a}_h(\Eb_h, \Eb_h)$, we obtain
\begin{align*}
&(\epsilon_0-m)\|  \Eb_h\|_{0,\Omega}^2 + C_0\sum_{f\in\mathcal{F}_h}\| \alpha^{1/2}\llbracket\Eb_{{h}}\rrbracket\|_{0,f}^2 + C_1 \sum_{K\in\mathcal{T}_h}\|\nabla_h\times \Eb_h\|_{0,K}^2	\\
&\leq \left(c_0|\Omega|^{1/2}+\|\fb\|_{\Vb^*}  \right)\|\Eb_h\|_h .
\end{align*}
Since $\epsilon_0-m>0$ and $\eta > \max\{1,\mu_1^2\}\,\tilde{C}^2/(2\,\mu_0^2)$, then
\begin{equation}
\|\Eb_h\|_h \leq \frac{c_0|\Omega|^{1/2}+\|\fb\|_{\Vb^*}}{\min\{\epsilon_0-m, C_0, C_1\}}.
\end{equation}
Therefore, $\|\Eb_h\|_h$ is bounded by a constant independent of $h$.
\end{proof}

\begin{remark}
We note that since $\epsilon_0 = k^{-1} \tilde{\epsilon}_0$, the condition $m<\epsilon_0$ can always be satisfied as long as the time step-size $k$ is sufficiently small.    
\end{remark}

\subsection{Existence and uniqueness}
In this subsection, we establish the existence and uniqueness of the solution to \eqref{eqn1.18} by following the approach developed in \cite{han2020minimization, han2022numerical}. To analyze the existence and uniqueness of the discrete solution later, we first recall the following results.
\begin{lemma}[{\cite[Theorem~3.4]{fan2003generalized}}]
\label{lem:2.2}
Let $V$ be a real Banach space, and let $g\colon V \to \mathbb{R}$ be locally Lipschitz continuous.
Then $g$ is strongly convex on $V$ with a constant $\alpha>0$ if and only if $\partial g$ 
is strongly monotone on $V$ with a constant $2\alpha$, i.e.,
\[
\langle \xi - \eta,\, u - v\rangle \;\ge\; 2\,\alpha \,\|u - v\|_{V}^2 
\quad
\forall\,u,v\in V,\;\xi\in \partial g(u),\;\eta\in \partial g(v).
\]
\end{lemma}

\begin{proposition}[{\cite[Proposition~2.5]{han2020minimization}}]
\label{prop:2.3}
Let $V$ be a real Hilbert space, and let $g\colon V \to \mathbb{R}$ be a locally Lipschitz continuous, strongly convex functional on $V$ with constant $\alpha>0$. Then there exist two constants 
$\overline{c}_0$ and $\overline{c}_1$ such that
\begin{equation}
\label{eq:2.11}
g(v) \;\ge\; \alpha\,\|v\|_{V}^2 \;+\; \overline{c}_0 \;+\; \overline{c}_1\,\|v\|_{V}
\quad
\forall\,v\in V.
\end{equation}
Consequently, $g(\cdot)$ is coercive on $V$.
\end{proposition}

Define an energy functional 
\begin{equation}
\mathcal{E}(\vb_h) = \frac{1}{2}a_h(\vb_h, \vb_h)+\int_\Omega \psi(\xb,\vb_h(\xb)) \mathrm{d}\xb - \langle \fb, \vb_h \rangle, \quad \vb_h \in \Vb_h.
\end{equation}
We consider the minimization problem
\begin{equation}\label{minimization problem}
\Eb_h \in \Vb_h, \quad \mathcal{E}(\Eb_h) = \inf\{\mathcal{E}(\vb_h) \mid \vb_h\in\Vb_h\}.
\end{equation}

\begin{lemma}
Assume \eqref{eqn1.other}, $m<\epsilon_0$ and $\eta > \max\{1,\mu_1^2\}\,\tilde{C}^2/(2\,\mu_0^2)$. Then the functional $\mathcal{E}(\cdot)$ is locally Lipschitz continuous, strongly convex and coercive on $\Vb_h$.
\end{lemma}
\begin{proof}
The local Lipschitz continuity of $\mathcal{E}(\cdot)$ follows immediately from the assumptions on $\psi$.
Let us prove the strong convexity.
For this purpose, define a linear operator $A_h \colon \Vb_h \to \Vb_h^*$ by
\begin{equation}
\label{eq:4.3}
\langle A_h \ub_h, \vb_h\rangle \;=\; a_h(\ub_h,\vb_h)
\quad
\forall\,\ub_h,\vb_h \in \Vb_h.
\end{equation}
Applying Lemma~\ref{lemma:boundedness}, we obtain 
\begin{equation*}
\|A_h\ub_h\|_{\Vb_h^*} = \sup_{\|\vb_h\|_h \ne 0} \frac{|\langle A_h\ub_h, \vb_h\rangle|}{\|\vb_h\|_h} = \sup_{\|\vb_h\|_h \ne 0} \frac{|a_h(\ub_h, \vb_h)|}{\|\vb_h\|_h} \leq C_b \|\ub_h\|_h \quad \forall \ub_h \in \Vb_h.
\end{equation*}
By Lemma~\ref{lemma:stability}, 
\begin{equation*}
\begin{aligned}
\langle A_h\ub_h - A_h\vb_h, \ub_h - \vb_h \rangle = a_h(\ub_h - \vb_h, \ub_h - \vb_h) \geq C_s \|\ub_h-\vb_h\|_h^2 \quad \forall  \ub_h, \vb_h \in \Vb_h.
\end{aligned}
\end{equation*}
Hence $A_h \in \mathcal{L}(\Vb_h, \Vb_h^*)$ and it is strongly monotone. Define a functional $\Psi : \Lb^2(\Omega) \to \mathbb{R}$ by
\[
\Psi(\vb)
\;=\;
\int_\Omega \psi(\xb,\vb(\xb))\,\mathrm{d}\xb
\quad
\forall\,\vb \in \Lb^2(\Omega).
\]
Then, by Thm.~4.37 in \cite{migorski2013nonlinear}, under assumption (\ref{eqn1.other}), $\Psi$ is well defined, locally Lipschitz continuous on $\Lb^2(\Omega)$, and
\begin{equation}
\label{eq:4.4}
\partial \Psi(\vb)
\;\subset\;
\int_\Omega \partial \psi\bigl(\vb)\,\mathrm{d}\xb
\end{equation}
in the sense that for any $\xib \in \partial \Psi(\vb)$, there exists a function $\zetab \in \Lb^2(\Omega)$ such that $\zetab(\xb) \in \partial \psi\bigl(\xb,\vb(\xb)\bigr)$ for a.e.\ $\xb \in \Omega$ and
\[
\langle \xib,\,\wb\rangle_{\Lb^2(\Omega)\times \Lb^2(\Omega)}
\;=\;
\int_\Omega \zetab(\xb) \cdot \wb(\xb)\,\mathrm{d}\xb
\quad
\forall\,\wb \in \Lb^2(\Omega).
\]
For $\vb_h \in \Vb_h$ and $\etab \in \partial \mathcal{E}(\vb_h)$, by (\ref{subaddition of Clarke}) we can write
\begin{equation}
\label{eqn eta}
\etab = A_h \vb_h + \xib - \fb,
\quad
\xib \in \partial \Psi(\vb_h).
\end{equation}
Thus, for $i=1,2$, with $\vb_{h,i} \in \Vb_h$ and $\etab_i \in \partial \mathcal{E}(\vb_{h,i})$, by \eqref{eq:4.4} we have $\zetab_i \in \Lb^2(\Omega)$ such that $\zetab_i(\xb) \in \partial \psi\bigl(\xb,\vb_{h,i}(\xb)\bigr)$ for a.e.\ $\xb \in \Omega$ and
\[
\langle \etab_i, \wb\rangle
\;=\;
\langle A_h \vb_{h,i}, \wb\rangle
\;+\;
\int_\Omega \zetab_i(\xb)\cdot \wb(\xb)\,\mathrm{d}\xb
\;-\;
\langle \fb, \wb\rangle
\quad
\forall\,\wb \in \Lb^2(\Omega).
\]
Thus, from (\ref{property of psi 5}) and Lemma~\ref{lemma:stability},
\begin{equation*}
\begin{aligned}
\langle \etab_1 - \etab_2,\; \vb_{h,1} - \vb_{h,2}\rangle = &\langle A_h \vb_{h,1} - A_h\vb_{h,2}, \vb_{h,1} - \vb_{h,2}\rangle+\int_\Omega \bigl(\zetab_1 - \zetab_2\bigr)\cdot (\vb_{h,1} - \vb_{h,2})\,\mathrm{d}\xb
\\
\geq 
&(\epsilon_0-m)\|  \vb_{h,1} - \vb_{h,2}\|_{0,\Omega}^2 + C_0\sum_{f\in\mathcal{F}_h}\| \alpha^{1/2}\llbracket\vb_{h,1} - \vb_{h,2}\rrbracket\|_{0,f}^2 \\
&+C_1\sum_{K\in\mathcal{T}_h}\|\nabla_h\times (\vb_{h,1} - \vb_{h,2})\|_{0,K}^2
\\
\geq& \min\{\epsilon_0-m, C_0, C_1\} \| \vb_{h,1} - \vb_{h,2}\|^{2}_h.
\end{aligned}
\end{equation*}
Thus, by Lemma~\ref{lem:2.2}, $\mathcal{E}(\cdot)$ is strongly convex.
Moreover, by Proposition~\ref{prop:2.3}, $\mathcal{E}(\cdot)$ is coercive on~$\Vb_h$.
\end{proof}

\begin{proposition}
Under the assumptions \eqref{eqn1.other}, $m<\epsilon_0$ and $\eta > \max\{1,\mu_1^2\}\,\tilde{C}^2/(2\,\mu_0^2)$, the minimization problem \eqref{minimization problem} has a unique solution $\Eb_h \in \Vb_h$
\end{proposition}
\begin{proof}
Since $\mathcal{E}(\cdot)$ is continuous, strictly convex and coercive on $\Vb_h$, from \cite[\S3.3.2]{han2009theoretical} the minimization problem (\ref{minimization problem}) has
a unique solution.
\end{proof}
\begin{theorem}
Assume \eqref{eqn1.other}, $m<\epsilon_0$ and $\eta > \max\{1,\mu_1^2\}\,\tilde{C}^2/(2\,\mu_0^2)$. Then for any $\fb\in \Vb^*$, the problem \eqref{eqn1.18} has a unique solution $\Eb_h\in\Vb_h$, which is also the unique solution of the minimization problem \eqref{minimization problem}.
\end{theorem}
\begin{proof}
For the solution $\Eb_h\in\Vb_h$ of the minimization problem (\ref{minimization problem}), we apply (\ref{eqn eta}) to get
\[
\langle A_h\Eb_h, \vb_h \rangle + \int_\Omega \zetab(\xb)\cdot \vb_h(\xb) \mathrm{d}\xb - \langle \fb, \vb_h \rangle \geq 0,
\]
where $\zetab \in \Lb^2(\Omega)$ such that $\zetab(\xb) \in \partial \psi\bigl(\xb,\Eb_h(\xb)\bigr)$ for a.e.\ $\xb \in \Omega$. Using the property of the Clarke subdifferential (\ref{eqnclarke1}) we have
\[
\psi^0(\Eb_h(\xb); \vb_h(\xb)) \geq \zetab(\xb)\cdot \vb_h(\xb) \quad \text{a.e. } \xb \in \Omega.
\]
Combining the above two inequalities, we can see that $\Eb_h$ is a solution of the problem (\ref{eqn1.18}). 

Now, let us prove uniqueness of the solution. Assume that $\Eb_h, \tilde{\Eb}_h\in\Vb_h$ are two solutions of the problem (\ref{eqn1.18}). Then we have 
\begin{equation}\label{finite element problem 2}
a_h(\tilde{\Eb}_h,\vb_h)+\int_\Omega\psi^0(\tilde{\Eb}_h;\vb_h)\mathrm{d}\xb \geq\langle \fb,\vb_h\rangle\quad \forall \vb_h\in\Vb_h.
\end{equation}
Take $\vb_h=\tilde{\Eb}_h-\Eb_h$ in (\ref{eqn1.18}) and $\vb_h=\Eb_h-\tilde{\Eb}_h$ in (\ref{finite element problem 2}). Adding the two resulting inequalities and using Lemma \ref{lemma:stability}, we obtain 
\begin{equation*}
\begin{aligned}
&\epsilon_0\|  \tilde{\Eb}_h - \Eb_h\|_{0,\Omega}^2 + C_0\sum_{f\in\mathcal{F}_h}\| \alpha^{1/2}\llbracket\tilde{\Eb}_h - \Eb_h\rrbracket\|_{0,f}^2 +C_1\sum_{K\in\mathcal{T}_h}\|\nabla_h\times (\tilde{\Eb}_h - \Eb_h)\|_{0,K}^2 
\\ 
& \leq a_h(\tilde{\Eb}_h-\Eb_h,\tilde{\Eb}_h-\Eb_h) \leq \int_{\Omega} \left(\psi^0(\Eb_h;\tilde{\Eb}_h - \Eb_h) + \psi^0(\tilde{\Eb}_h; \Eb_h-\tilde{\Eb}_h) \right) \mathrm{d}\xb 
\\
& \leq m \| \tilde{\Eb}_h - \Eb_h \|_{0,\Omega}^2.
\end{aligned}
\end{equation*}
By the smallness condition $m<\epsilon_0$, we deduce that $\tilde{\Eb}_h = \Eb_h$.
\end{proof}

\subsection{A priori error estimate}
In this subsection, we derive a priori error estimates for the IPDG method. To this end, we first present the following lemma, which provides approximation estimates for the second-family N\'{e}d\'{e}lec interpolation operator $\Pib_N$ \cite{monk2003finite,nedelec1980mixed,nedelec1986new}.

\begin{lemma}\label{lemma5.4} 
Assume that $\{\mathcal{T}_h\}$ is a shape-regular family of tetrahedral or affine hexahedral (including parallelepiped) partitions of $\overline{\Omega}$, and assume $\Eb \in \Hb_0(\curlb; \Omega) \cap \Hb^s(\Omega)$ with $ \nabla\times \Eb \in \Hb^s(\Omega)$, where $s > 1/2$. Then the following error estimates hold:
\begin{align}
\|\Eb - \Pib_N \Eb\|_{\curlb,\Omega} & \leq C_N h^{\min\{s, l\}} (\|\Eb\|_{s, \Omega} + \|\nabla\times \Eb\|_{s, \Omega}),
\label{eqn:lemma4-1} \\
\|\Eb - \Pib_N \Eb\|_{h} & \leq C_N h^{\min\{s, l\}} (\|\Eb\|_{s, \Omega} + \|\nabla\times \Eb\|_{s, \Omega}),
\label{energy norm estimate}
\end{align}
where $C_N > 0$ is a constant depending on the mesh regularity and the polynomial degree $l$ but independent of $h$, and for \eqref{energy norm estimate}, $C_N$ also depends on the upper and lower bounds of the coefficients $\mu$ and $\epsilon$.
 
Moreover, if $\Eb \in \Hb_0(\curlb; \Omega) \cap \Hb^{s+1}(\Omega)$ for some number $s > 0$, then
\begin{equation}
 \|\Eb - \Pib_N \Eb\|_{0, \Omega} \leq C_N h^{\min\{s, l\}+1} \|\Eb\|_{s+1, \Omega}. 
 \label{eqn:lemma4-3}
 \end{equation} 
\end{lemma}
 
The estimate \eqref{eqn:lemma4-1} can be found in \cite{monk2003finite}, and \eqref{eqn:lemma4-3} appears in \cite[Lemma 4.1]{houston2005interior}. The error bound \eqref{energy norm estimate} follows from \eqref{eqn:lemma4-1} because the coefficients $\epsilon$ and $\mu$ are bounded and, for $\Eb \in \Hb_0(\curlb;\Omega)$, the tangential jump $\llbracket\Eb - \Pib_N \Eb\rrbracket$ vanishes.

For $\vb\in\Vb(h)$, we define
\begin{equation}\label{eqn2.9}
r_h(\Eb;\vb)=\sum_{f\in\mathcal{F}_h}\int_f\llbracket\vb\rrbracket \cdot \left\{\mu^{-1}\nabla\times\Eb-\mu^{-1}\Pib_h(\nabla\times\Eb)\right\}\mathrm{d}S.
\end{equation}
For $r_h(\Eb; \vb)$ to be well-defined, the condition $\nabla\times \Eb \in \Hb^s(\Omega)$, where $s > 1/2$, is assumed. Under this condition, we have the following result \cite{grote2007interior}.

\begin{lemma}\label{lemma2.6} 
Assume $\nabla\times \Eb \in \Hb^s(\Omega)$, $s > 1/2$.  Then, 
\[ |r_h(\Eb;\vb)|\leq C_R h^{\min\{s,l+1\}}|\vb|_h\|\nabla\times\Eb\|_{s,\Omega}\quad\forall\,\vb \in \Vb(h),\]
where the constant $C_R$ is independent of the mesh size but depends on $\eta$, the upper and lower bounds of the coefficient $\mu$, the mesh regularity, and the polynomial order $l$. 
\end{lemma}

\begin{theorem}\label{theorem4.2}
Assume that ${\mathcal{T}_h}$ is a shape-regular family of tetrahedral or affine hexahedral (including parallelepiped) partitions of $\overline{\Omega}$. Let $\Eb$ and $\Eb_h$ be the solutions of \eqref{eqn3.17} and \eqref{eqn1.18}, respectively. Suppose that
\[
\mu^{-1}\nabla\times \Eb \in \Hb(\curlb,\Omega), 
\qquad
\left.\left(\mu^{-1}\nabla\times \Eb\right)\right|_K \in \Hb^1(K)
\quad \text{for all } K\in\mathcal{T}_h.
\]
Assume that the penalty parameter $\eta$ satisfies
\[
\eta > \frac{\max\{1,\mu_1^2\}\,\tilde{C}^2}{2\,\mu_0^2},
\]
and that $m<\epsilon_0$.

If $\Eb \in \Hb^{s}(\Omega)$ and $\nabla\times\Eb \in \Hb^{s}(\Omega)$ with $s>1/2$, then there exists a positive constant $C$, independent of $h$, such that
\begin{equation}\label{error_bd2}
\| \Eb-\Eb_h\|_h \le Ch^{\min\{s,l\}/2}.
\end{equation}
Moreover, if $\Eb \in \Hb^{s+1}(\Omega)$ with $s>1/2$, then
\begin{equation}
\| \Eb-\Eb_h\|_h \le Ch^{\left(\min\{s,l\}+1\right)/2}.
\label{error_bd}
\end{equation}
\end{theorem}

\begin{proof}
Denote $\Eb_I=\Pib_N\Eb$ and write
\begin{equation}\label{eqn2.10}
\widetilde{a}_h(\Eb_I-\Eb_h,\Eb_I-\Eb_h)=T_1+T_2,
\end{equation}
where $T_1=\widetilde{a}_h(\Eb_I-\Eb,\Eb_I-\Eb_h)$ and $T_2=\widetilde{a}_h(\Eb-\Eb_h,\Eb_I-\Eb_h)$. Using Young’s inequality \eqref{mCS} and the boundedness estimate for $\widetilde{a}_h$, we have, for any small $\varepsilon>0$,
\begin{equation}\label{eqn2.11}
T_1\leq C_b\|\Eb_I-\Eb\|_h \|\Eb_I-\Eb_h\|_h  \leq \frac{\varepsilon}{4}\|\Eb_I-\Eb_h\|_h^2+\frac{C_b^2}{\varepsilon}\|\Eb_I-\Eb\|_h^2.
\end{equation}
Note that on $\mathcal{F}_h$, $\llbracket\Eb\rrbracket=\zerob$, $\{\Eb\}=\Eb$, and $ \llbracket\nabla\times\Eb\rrbracket=\zerob$. Thus,
\begin{equation}\label{eq: tilde a_h(E,EI, Eh)}
\begin{aligned}
\widetilde{a}_h(\Eb,\Eb_I-\Eb_h)=&\int_\Omega\epsilon \Eb \cdot (\Eb_I-\Eb_h) \,\mathrm{d}\xb+\sum_{K\in\mathcal{T}_h}\int_K\mu^{-1} (\nabla\times \Eb) \cdot (\nabla_h\times(\Eb_I-\Eb_h))\,\mathrm{d}\xb
\\
&-\sum_{f\in\mathcal{F}_h}\int_f\llbracket\Eb_I-\Eb_h\rrbracket\cdot \{\mu^{-1}\Pib_h(\nabla\times\Eb)\}\,\mathrm{d}S
\\
= &a_h\left(\Eb, \Eb_I-\Eb_h\right)+\sum_{f\in\mathcal{F}_h}\int_f\llbracket\Eb_I-\Eb_h\rrbracket\cdot \{\mu^{-1}\nabla\times\Eb\}\,\mathrm{d}S 
\\ 
&-\sum_{f\in\mathcal{F}_h}\int_f\llbracket\Eb_I-\Eb_h\rrbracket\cdot \{\mu^{-1}\Pib_h(\nabla\times\Eb)\}\,\mathrm{d}S.
\end{aligned}
\end{equation}
Since $\Eb_h-\Eb_I \in \Vb_h$, it follows from Theorem~\ref{thm: consistency} that
\begin{equation*}
a_h(\Eb,\Eb_h-\Eb_I)+\int_\Omega\psi^0(\Eb;\Eb_h-\Eb_I)\,\mathrm{d}\xb\ge\langle\fb,\Eb_h-\Eb_I\rangle,
\end{equation*}
that is,
\begin{equation}\label{eq: consistency2}
\langle\fb,\Eb_I-\Eb_h\rangle+\int_\Omega\psi^0(\Eb;\Eb_h-\Eb_I)\,\mathrm{d}\xb \geq a_h(\Eb,\Eb_I-\Eb_h).
\end{equation}
Combining \eqref{eq: tilde a_h(E,EI, Eh)} with \eqref{eq: consistency2}, we obtain
\begin{equation}\label{eqn2.12}
\begin{aligned}
\widetilde{a}_h(\Eb,\Eb_I-\Eb_h)\leq & \langle\fb,\Eb_I-\Eb_h\rangle+\int_\Omega\psi^0(\Eb;\Eb_h-\Eb_I)\,\mathrm{d}\xb
\\
& + \sum_{f\in\mathcal{F}_h}\int_f\llbracket\Eb_I-\Eb_h\rrbracket \cdot \left\{\mu^{-1}\nabla\times\Eb-\mu^{-1}\Pib_h(\nabla\times\Eb)\right\}\,\mathrm{d}S 
\\
= & \langle\fb,\Eb_I-\Eb_h\rangle + \int_\Omega \psi^0(\Eb;\Eb_h-\Eb_I)\,\mathrm{d}\xb+r_h(\Eb;\Eb_I-\Eb_h).
\end{aligned}
\end{equation}
Letting $\vb_h=\Eb_I-\Eb_h$ in (\ref{eqn1.18}), we get
\begin{equation}\label{eqn2.13}
	-\widetilde{a}_h(\Eb_h,\Eb_I-\Eb_h)=-a_h(\Eb_h,\Eb_I-\Eb_h)\leq \int_\Omega\psi^0(\Eb_h;\Eb_I-\Eb_h)\mathrm{d}\xb-\langle\fb,\Eb_I-\Eb_h\rangle.
\end{equation}
Combining (\ref{eqn2.12}) and (\ref{eqn2.13}), and using the subadditivity of the generalized directional derivative, we obtain
\begin{equation}
\begin{aligned}
T_2 \leq&  \int_\Omega \psi^0(\Eb;\Eb_h-\Eb_I)\mathrm{d}\xb +\int_\Omega\psi^0(\Eb_h;\Eb_I-\Eb_h)\mathrm{d}\xb + r_h(\Eb;\Eb_I-\Eb_h) \\
\leq 
&\int_\Omega \psi^0(\Eb;\Eb_h-\Eb)\mathrm{d}\xb+\int_\Omega \psi^0(\Eb;\Eb-\Eb_I)\mathrm{d}\xb
\\
&+\int_\Omega\psi^0(\Eb_h;\Eb_I-\Eb)\mathrm{d}\xb+\int_\Omega\psi^0(\Eb_h;\Eb-\Eb_h)\mathrm{d}\xb+r_h(\Eb;\Eb_I-\Eb_h).
\end{aligned}
\end{equation}
Using (\ref{eqn1.other})\,(d), we obtain
\begin{align}\notag
\int_\Omega \psi^0(\Eb;\Eb_h-\Eb)\mathrm{d}\xb+\int_\Omega\psi^0(\Eb_h;\Eb-\Eb_h)\mathrm{d}\xb\leq m\|\Eb-\Eb_h\|^2_{0,\Omega}.
\end{align}
By the triangle inequality $\|\Eb-\Eb_h\|_{0,\Omega}\le\|\Eb-\Eb_I\|_{0,\Omega}+\|\Eb_I-\Eb_h\|_{0,\Omega}$ and Young’s inequality \eqref{mCS},
\[ \|\Eb-\Eb_h\|^2_{0,\Omega}\le (1+\varepsilon)\|\Eb_I-\Eb_h\|_{0,\Omega}^2+(1+1/\varepsilon)\|\Eb-\Eb_I\|_{0,\Omega}^2.\]
Therefore,
\begin{equation*}
\begin{aligned}
&\int_\Omega \psi^0(\Eb;\Eb_h-\Eb)\mathrm{d}\xb+\int_\Omega\psi^0(\Eb_h;\Eb-\Eb_h)\mathrm{d}\xb
\\
\leq&(1+\varepsilon) m\|\Eb_I-\Eb_h\|_{0,\Omega}^2+(1+1/\varepsilon)m\|\Eb-\Eb_I\|_{0,\Omega}^2.
\end{aligned}
\end{equation*}
Using (\ref{eqn1.other2}), we have
\begin{align}
\int_\Omega \psi^0(\Eb;\Eb-\Eb_I)\mathrm{d}\xb \leq\int_\Omega(c_0+c_1|\Eb|)|\Eb-\Eb_I|\mathrm{d}\xb,	\\
\int_\Omega \psi^0(\Eb_h;\Eb_I-\Eb)\mathrm{d}\xb \leq\int_\Omega(c_0+c_1|\Eb_h|)|\Eb-\Eb_I|\mathrm{d}\xb.
\end{align}
Since $\Omega$ is bounded, and since $\Eb_h$ is uniformly bounded, the Cauchy-Schwarz inequality yields constants $C_{Eb},C_{Ehb}>0$ such that
\begin{equation*}
\int_\Omega(c_0+c_1|\Eb|)|\Eb-\Eb_I|\mathrm{d}\xb \leq \left(\int_\Omega(c_0+c_1|\Eb|)^2\mathrm{d}\xb\right)^{1/2} \| \Eb-\Eb_I\|_{0,\Omega}= C_{Eb} \| \Eb-\Eb_I\|_{0,\Omega},
\end{equation*}
\begin{equation*}
\int_\Omega(c_0+c_1|\Eb_h|)|\Eb-\Eb_I|\mathrm{d}\xb \leq \left(\int_\Omega(c_0+c_1|\Eb_h|)^2\mathrm{d}\xb\right)^{1/2} \| \Eb-\Eb_I\|_{0,\Omega}\le C_{Ehb} \| \Eb-\Eb_I\|_{0,\Omega}.
\end{equation*}
Therefore,
\begin{align}
T_2&\leq (C_{Eb}+C_{Ehb})\|\Eb-\Eb_I\|_{0,\Omega}+(1+\varepsilon)m\|\Eb_I-\Eb_h\|_{0,\Omega}^2+(1+1/\varepsilon)m\|\Eb-\Eb_I\|_{0,\Omega}^2
\nonumber\\
&\quad{} +r_h(\Eb;\Eb_I-\Eb_h).\label{eqn2.20}
\end{align}
Since $\Eb\in\Hb^{s+1}(\Omega)$ implies $\nabla\times\Eb\in\Hb^{s}(\Omega)$, the regularity assumption imposed on $\Eb$ in the theorem already ensures that $\nabla\times\Eb\in\Hb^{s}(\Omega)$ for $s>1/2$. Combining (\ref{eqn2.10}), (\ref{eqn2.11}), (\ref{eqn2.20}), Lemma \ref{lemma:stability} and Lemma \ref{lemma2.6}, we obtain
\begin{equation}\label{eqn:error bound}
\begin{aligned}
&\epsilon_0\|  \Eb_I-\Eb_h\|_{0,\Omega}^2 + C_0\sum_{f\in\mathcal{F}_h}\| \alpha^{1/2}\llbracket\Eb_I-\Eb_h\rrbracket\|_{0,f}^2 + C_1 \sum_{K\in\mathcal{T}_h}\|\nabla_h\times (\Eb_I-\Eb_h)\|_{0,K}^2 
\\
& \leq  (C_{Eb}+C_{Ehb})\|\Eb-\Eb_I\|_{0,\Omega}+(1+\varepsilon)m \|\Eb_I-\Eb_h\|_{0,\Omega}^2+(1+1/\varepsilon)m\|\Eb-\Eb_I\|_{0,\Omega}^2 \\
& \quad +\frac{\varepsilon}{4}\|\Eb_I-\Eb_h\|_h^2+\frac{C_b^2}{\varepsilon}\|\Eb_I-\Eb\|_h^2+C_R h^{\min\{s,l+1\}}|\Eb_I-\Eb_h|_h\|\nabla\times \Eb\|_{s,\Omega}.
\end{aligned}
\end{equation}
Using Young’s inequality~\eqref{mCS}, we have, for any $\varepsilon>0$,
\begin{equation}
	C_R h^{\min\{s,l+1\}}|\Eb_I-\Eb_h|_h \|\nabla\times\Eb\|_{s,\Omega} \leq \frac{\varepsilon}{4}|\Eb_I-\Eb_h|_h^2+\frac{C_R^2 h^{2\min\{s,l+1\}}}{\varepsilon}\|\nabla\times\Eb\|_{s,\Omega}^2.
\end{equation}
Since $m<\epsilon_0$, we can choose a sufficiently small $\varepsilon>0$ such that
$\epsilon_0 - m - (m+1/4)\varepsilon>0$ and $\min\{C_0, C_1\} > \varepsilon/4$.  
When $\Eb \in \Hb^{s}(\Omega)$ and $\nabla\times\Eb \in \Hb^{s}(\Omega)$ with $s>1/2$, the term $\|\Eb-\Eb_I\|_{0,\Omega}$ on the right-hand side of \eqref{eqn:error bound} does not yield the optimal convergence order. In this case, it can only be controlled by \eqref{energy norm estimate}. Invoking the estimate \eqref{energy norm estimate}, we deduce from \eqref{eqn:error bound} that
\[
\|\Eb_I-\Eb_h\|_h\leq C\,h^{\min\{s,l\}/2}.
\]
Moreover, when $\Eb \in \Hb^{s+1}(\Omega)$, the term $\|\Eb-\Eb_I\|_{0,\Omega}$ attains the optimal convergence order given by \eqref{eqn:lemma4-3}. Therefore, we obtain
\[
\|\Eb_I-\Eb_h\|_h\leq C\,h^{\left(\min\{s,l\}+1\right)/2}.
\]

Finally, we use the triangle inequality $\|\Eb-\Eb_h\|_h\le\|\Eb-\Eb_I\|_h +\|\Eb_I-\Eb_h\|_h$ to conclude the error bound ~\eqref{error_bd2} and ~\eqref{error_bd}.
\end{proof}

In particular, when linear elements are used and the solution $\Eb$ of the hemivariational inequality~\eqref{eqn3.17} satisfies
\[
\Eb \in \Hb^2(\Omega),
\]
we obtain the optimal first-order convergence estimate:
\[
\| \Eb-\Eb_h\|_h \leq C\,h.
\]

\section{Numerical example}
This paper primarily focuses on the three-dimensional case. However, a two-dimensional vector field $\Eb(x, y) = (E_1(x, y), E_2(x, y))$ can be embedded into three dimensions by setting 
\[
\Eb(x, y, z) = (E_1(x, y), E_2(x, y), 0).
\]
Similarly, a two-dimensional normal vector $\nb = (n_1, n_2)$ can be identified with $\nb = (n_1, n_2, 0)$ in $\mathbb{R}^3$. Under this identification, we have
\[ \nabla\times \Eb = \left(0, 0, \frac{\partial E_2}{\partial x} - \frac{\partial E_1}{\partial y}\right),
\quad \nb \times \Eb = \left(0, 0, n_1 E_2 - n_2 E_1\right). 
\]
Therefore, the three-dimensional analysis presented in this paper applies directly to the two-dimensional case. In this section, we report numerical results for a two-dimensional example.

Let $\Omega=(0,1)^2$, $\epsilon=1$, $\mu=1$, and the source term
\begin{equation*}
\fb=\left(
\begin{array}{c}
	(1+2\pi^2)\cos(\pi x)\sin(\pi y)\\
	-(1+2\pi^2)\sin(\pi x)\cos(\pi y)\\
\end{array}
\right).
\end{equation*}
The function $\psi$ is chosen as follows:
\begin{equation*}
\omega(t)=(a-b)e^{-\beta t}+b, \quad \psi(\Eb)=\int_0^{|\Eb|} \omega(t) \mathrm{d}t,
\end{equation*}
where $a > b > 0$ and $\beta > 0$. It can be verified that this function $\psi$ satisfies the conditions in (\ref{eqn1.other}), where the parameter $m$ in (\ref{eqn1.other})\,(d) is given by $m = \beta(a - b)$. We choose the parameters $a=0.004$, $b=0.002$, and $\beta=100$ in the function $\omega(t)$.

To solve the problem \eqref{eqn1.18}, we employ the Uzawa iterative algorithm (cf.\ \cite{barboteu2013analytical}). It is stated as Algorithm \ref{Uzawa iterative}. We choose the penalty parameter $\eta=10^3$. Since the analytic solution of the inequality problem is unknown, we will use the numerical solution with a grid size of $h=2^{-6}$ as a reference solution to compute the error of the numerical solution with coarser grid sizes.

\begin{algorithm}
\caption{Uzawa iteration for the $\Hcurlb$-elliptic HVI}
\label{Uzawa iterative}
\KwIn{Maximal number of iteration steps $l_{\rm max}$, error tolerance $\mathrm{tol}>0$.}
\KwOut{The numerical solution of the inequality problem $\Eb^*$.}
Find $\Eb^0_h \in \Vb_h$ such that $a_h(\Eb_h^0,\vb_h)=\langle \fb,\vb_h\rangle\quad \forall\, \vb_h\in\Vb_h$\;
\For{{$l = 1$ \textbf{ to } $l_{\rm max}$}}{
Find $\Eb_h^l \in \Vb_h$ such that $\displaystyle a_h(\Eb_h^l,\vb_h)+\int_\Omega \boldsymbol{\lambda}^l_h \cdot \vb_h =\langle \fb,\vb_h\rangle\quad \forall \vb_h\in\Vb_h$, where $\boldsymbol{\lambda}^l_h\in \omega(|\Eb^{l-1}_h|)\partial|\Eb^l_h|$\;
\If{$\| \Eb^{l}_h-\Eb^{l-1}_h\|_{0,\Omega} \leq \mathrm{tol} \|\Eb^{l-1}_h\|_{0,\Omega}$ \,\, and \, \,$\| \boldsymbol{\lambda}^{l}_h-\boldsymbol{\lambda}^{l-1}_h\|_{0,\Omega} \leq \mathrm{tol} \|\boldsymbol{\lambda}^{l-1}_h\|_{0,\Omega}$}
{
break
}
}
$\Eb^*=\Eb_h^l$\;
\Return{$\Eb^*$}\;
\end{algorithm}

Numerical results are reported in Table \ref{table2}.  It is observed that for this example, the numerical solutions exhibit second-order convergence in the $L^2$-norm and first-order convergence in the energy norm $\|\cdot\|_h$ with respect to the grid size $h$. The observed convergence rate in the energy norm $\|\cdot\|_h$ agrees with the theoretical result established in Theorem \ref{theorem4.2}. Figure \ref{Fig5} illustrates these convergence rates. Figure \ref{Fig4} displays the streamline plots of the numerical solutions for the inequality problem with grid sizes $h = 2^{-4}$ and $h = 2^{-5}$.

\begin{table}[H]
\centering
\caption{Errors and convergence orders of the numerical solutions}
\label{table2}
\setlength{\tabcolsep}{7.5mm}{
	\begin{tabular}{ccccc}
		\toprule
		$h$&$\| \Eb-\Eb_h\|_{0,\Omega}$&Order&$\| \Eb-\Eb_h\|_h$&Order  \\
		\midrule
		$2^{-1}$&2.1929e-1&-&1.5957&-  \\
		$2^{-2}$&6.0856e-2&1.8494&8.3933e-1&0.9269 \\
		$2^{-3}$&1.5644e-2&1.9598&4.2970e-1&0.9659  \\
		$2^{-4}$&3.8917e-3&2.0071&2.1416e-1&1.0047  \\
		\bottomrule
\end{tabular}}
\end{table}

\begin{figure}
    \centering
    \includegraphics[width=0.7\linewidth]{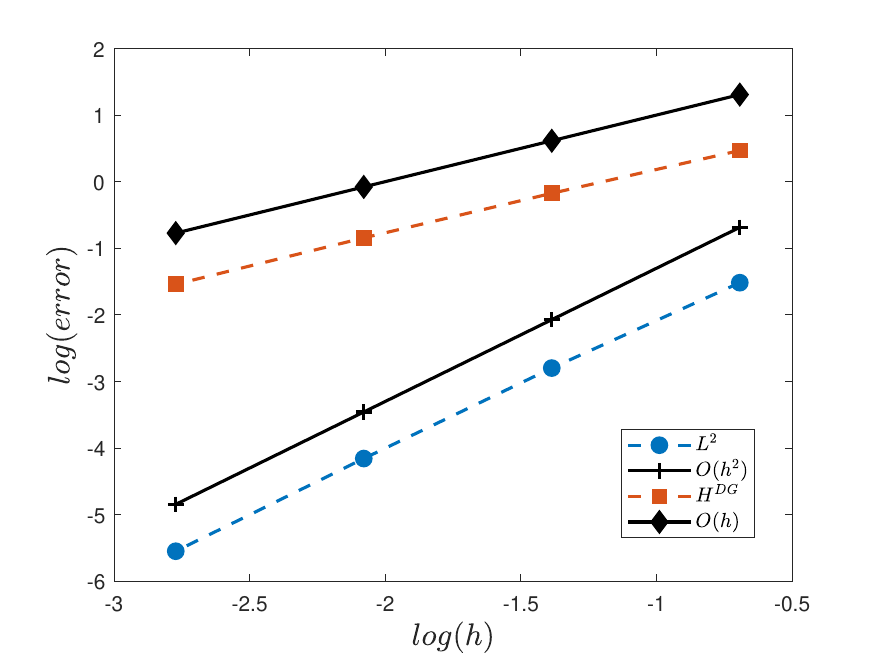}
    \caption{Convergence of the numerical solutions}
    \label{Fig5}
\end{figure}

\begin{figure}[H]
    \subcaptionbox{$\Eb_h$ for $h = 2^{-4}$}{\includegraphics[width=0.45\textwidth]{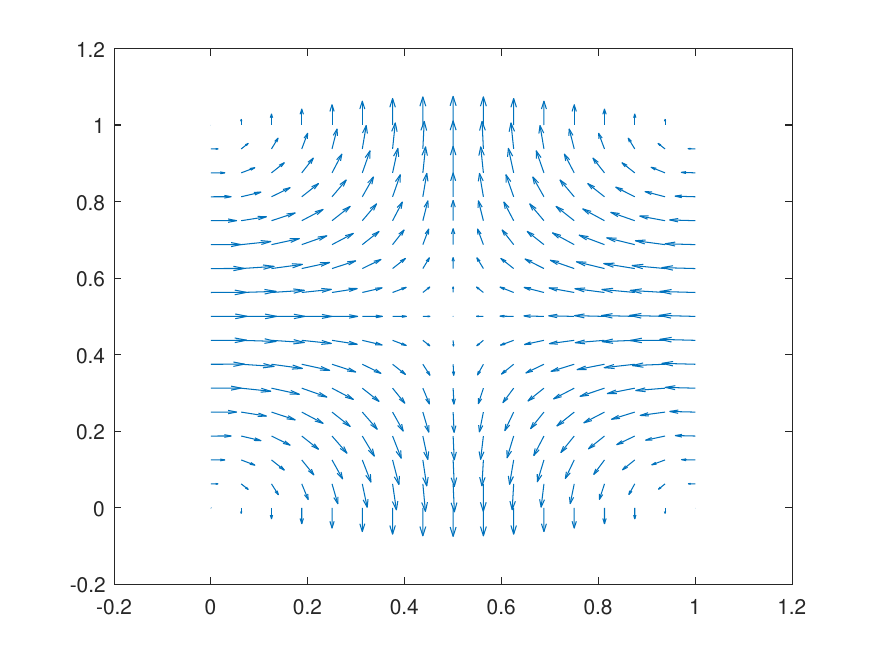}}
    \hspace{0.05\textwidth}
    \subcaptionbox{$\Eb_h$ for $h = 2^{-5}$}{\includegraphics[width=0.45\textwidth]{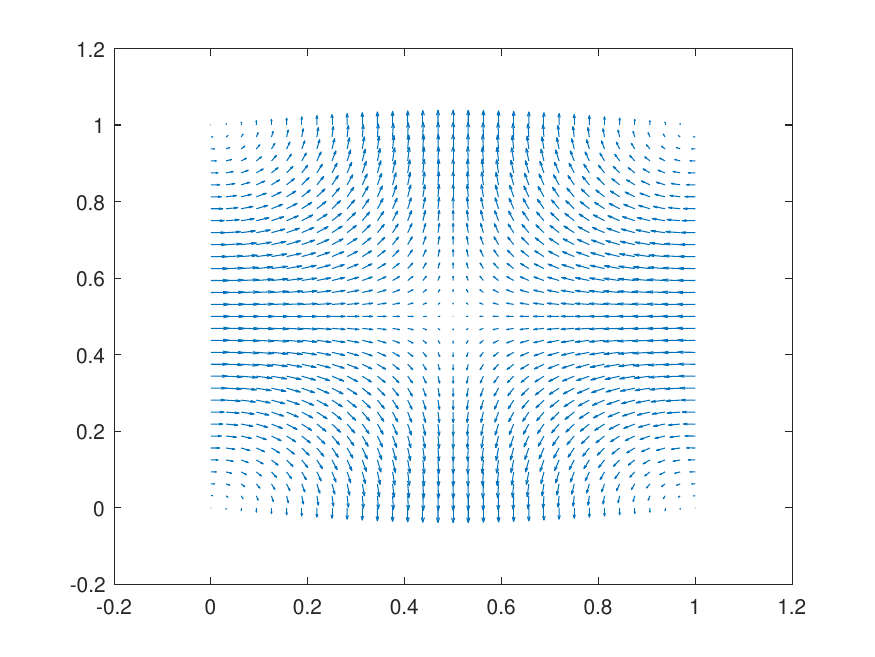}}
    \caption{Streamlines of numerical solutions}
    \label{Fig4}
\end{figure}

\bigskip

\noindent {\bf Data Availability Statement.} No data was used for the research described in the article.
\bigskip

\noindent {\bf Declaration Statement.}
The authors have no conflicts of interest to declare that are relevant to the content of this article.
\bigskip

\noindent {\bf Acknowledgments}. We thank the two anonymous referees for their valuable comments and suggestions
on the original version of the manuscript. 

%--- Section ---%
%-------------------------------------------
% References
%-------------------------------------------

% Print bibliography
%\newpage
\printbibliography

\end{document}